\documentclass{imsart}
 
\RequirePackage[OT1]{fontenc}
\RequirePackage{amsthm,amsmath}
\RequirePackage[numbers]{natbib}
\RequirePackage[colorlinks,citecolor=blue,urlcolor=blue]{hyperref}

\usepackage{amsfonts, amssymb}
\usepackage{graphicx, color, latexsym}

\usepackage{esint}
\usepackage{dsfont}

\newtheorem{lem}{Lemma}
\theoremstyle{remark}
\newtheorem{remark}{Remark}

{}
\def\1{1\!{\rm l}}

\newcommand{\leqa}{\lesssim}
\newcommand{\geqa}{\gtrsim}

\newcommand{\EM}{\ensuremath}

\newcommand{\al}{\alpha}
\newcommand{\be}{\beta}

\newcommand{\ga}{\gamma}
\newcommand{\Ga}{\Gamma}

\newcommand{\La}{\Lambda}

\newcommand{\ta}{\tau}
\newcommand{\veps}{\varepsilon}
\newcommand{\vphi}{\varphi}

\newcommand{\cA}{\EM{\mathcal{A}}}
\newcommand{\cB}{\EM{\mathcal{B}}}
\newcommand{\cC}{\EM{\mathcal{C}}}
\newcommand{\cD}{\EM{\mathcal{D}}}
\newcommand{\cE}{\EM{\mathcal{E}}}
\newcommand{\cF}{\EM{\mathcal{F}}}

\newcommand{\cH}{\EM{\mathcal{H}}}

\newcommand{\cJ}{\EM{\mathcal{J}}}

\newcommand{\cL}{\EM{\mathcal{L}}}
\newcommand{\cM}{\EM{\mathcal{M}}}
\newcommand{\cN}{\EM{\mathcal{N}}}
\newcommand{\cP}{\EM{\mathcal{P}}}

\newcommand{\cY}{\EM{\mathcal{Y}}}

\newcommand{\psg}{{\langle}}
\newcommand{\psd}{{\rangle}}

\definecolor{blendedblue}{rgb}{0.2,0.2,0.7}

\DeclareMathAlphabet{\mathpzc}{OT1}{pzc}{m}{it}

\newcommand{\noi}{\noindent}

\newcommand{\RR}{\mathbb{R}}

\newcommand{\N}{\mathbb{N}}

\newcommand{\given}{\,|\,}

\newcommand{\rn}{\sqrt{n}}

\newcommand{\bet}{\operatorname{Beta}}
\newcommand{\bel}{\operatorname{Be}}
\newcommand{\kl}{\operatorname{KL}}

\newcommand{\bi}{\begin{enumerate}[label=\roman*)]}
\newcommand{\ei}{\end{enumerate}}
\newcommand{\ba}{\begin{array}{rcl}}
\newcommand{\ea}{\end{array}}
\newcommand{\di}{\displaystyle}

\newcommand{\mockalph}[1]{}

\startlocaldefs
\theoremstyle{plain}
\newtheorem{thm}{Theorem}
\endlocaldefs

\begin{document}
\begin{frontmatter}
\title{Spike and slab {P}\'{o}lya tree posterior densities: adaptive inference}
\runtitle{Adaptive P\'olya trees}
\thankstext{T1}{Work partly supported by the grant ANR-17-CE40-0001-01
of the French National Research Agency ANR (project BASICS). Results in Sections 2, 3.1 and 3.2 correspond to Chapter 4 of R.M.'s PhD thesis, which was funded by Universit\'e Paris--Diderot.}

\begin{aug} 
\author{\fnms{Isma\"el} \snm{Castillo}
\ead[label=e1]{ismael.castillo@upmc.fr}}
\and
\author{\fnms{Romain} \snm{Mismer}
\ead[label=e2]{rmismer@free.fr}}

\address{Sorbonne Universit\'e and Universit\'e Paris--Diderot\\
Laboratoire Probabilit\'es, Statistique et et Mod\'elisation\\
4, place Jussieu\\
75005 Paris, France\\
\printead{e1,e2}}

\runauthor{I. Castillo and R. Mismer}
\runtitle{Spike and Slab P\'olya tree posterior distributions}

\affiliation{}

\end{aug}
 
\begin{abstract}
In the density estimation model, the question of adaptive inference using P\'olya tree--type prior distributions is considered. A class of prior densities having a tree structure, called spike--and--slab P\'olya trees, is introduced. For this class, 
two types of results are obtained: first, the Bayesian posterior distribution is shown to converge at the minimax rate for the supremum norm in an adaptive way, for any H\"older regularity of the true density between $0$ and $1$, thereby providing adaptive counterparts to the results for classical P\'olya trees in \cite{castillo2017}. 
Second, the question of uncertainty quantification is considered. An adaptive nonparametric Bernstein--von Mises theorem is derived. Next, it is shown that, under a self-similarity condition on the true density, certain credible sets from the posterior distribution are adaptive confidence bands, having prescribed coverage level and with a diameter shrinking at optimal rate in the minimax sense.
\end{abstract}




\end{frontmatter}

\section{Introduction}

P\'olya trees (abbreviated as PTs in the sequel) are a class of random probability distributions, where mass is spread following a dyadic tree structure. PTs were introduced in the 1960's and have been since then frequently used in Bayesian nonparametric statistics  as a popular choice of  prior distribution, for inference on probability distributions and densities. We refer to \cite{gvbook17},  Chapter 3 (see also \cite{lavine}) for an overview on PTs and their link to P\'olya urns \cite{polya30}, which was established in \cite{mauldinetal92}, who also coined the term `P\'olya tree'.  Although as such PTs are routinely used as part of statistical algorithms, not much was known until recently on their mathematical properties.

This work is a continuation of the paper \cite{castillo2017}, where it is shown that, for well chosen parameters, P\'olya trees are able to model smooth functions and  induce posterior distributions with optimal convergence rates in the minimax sense for a range of H\"older regularities. Previously known results for PTs  only dealt with consistency instead of rates, and it was not clear how to chose the parameters of the Beta distributions appearing in the construction of the PT to achieve optimal convergence rates; \cite{castillo2017} shows how a proper tuning of the Beta parameters leads to an optimal supremum-norm contraction rate for the posterior distribution. However, this choice depends on the H\"older regularity of the true unknown density, which is typically unknown. One goal of the present paper is to show how the prior distribution can be modified to yield optimal {\em adaptive} rates, that is rates that automatically adapt to the unknown smoothness parameter.

While optimal convergence rates are obviously desirable, in practice one often wants to know how much `confidence' there is in a given estimate, and this can typically be achieved by reporting a confidence set, whose diameter should then ideally be as small as possible. This goal is often referred to as {\em uncertainty quantification} and achieving it has recently been the object of much activity, in particular through the use of Bayesian methods, see e.g. the discussion paper \cite{svv15} and references therein. 
One aspect of the question is to study the limiting shape of the posterior distribution,  for instance by deriving so-called Bernstein--von Mises (BvM) theorems. This has recently be shown to be doable in nonparametric settings including density estimation in \cite{castillo2013, castillo2014} and was investigated for P\'olya trees in the fixed regularity case in \cite{castillo2017}, where  a Bernstein-von Mises theorem as well as a Donsker-type theorem for the distribution function were derived. Very recently, a first adaptive nonparametric BvM was obtained by Ray in \cite{ray2017} for the Gaussian white noise model. Here we shall consider the question of deriving an adaptative BvM in the density estimation model using P\'olya tree priors. Going further, we derive honest and adaptive confidence bands for the true unknown density under an (essentially unavoidable) self-similarity condition. 
 
\subsection{Density estimation model}

Suppose one observes a sample $X^{(n)}=(X_1,\cdots,X_n)$ of independent identically distributed (i.i.d.) variables of distribution $P_0=P_{f_0}$ admitting a density $f_0$ with respect to Lebesgue measure on the unit interval $[0,1]$.

The goal is to make inference on the unknown $f_0$ both in terms of estimation and confidence regions in a sense to be made precise below. To do so, we follow a Bayesian approach: to build the estimator, one considers the Bayesian model
\begin{align}
X_1,\ldots,X_n \given f \ & \sim\ P_f^{\otimes n}  \label{bayesdensity}\\
f \qquad& \sim \ \Pi,
\end{align}
where $\Pi$ is a distribution on densities on $[0,1]$ called {\em prior} distribution, and that below is chosen to be the law induced by certain random dyadic histograms built following a tree structure. The previous display specifies a Bayesian joint distribution of $(X^{(n)},f)$ from which one can deduce the {\em posterior distribution}, which is the conditional distribution $f\given X^{(n)}$ and will be denoted $\Pi[\cdot\given X]$ for short. This (generalised) estimator is a data-dependent distribution, which we study following a so-called frequentist analysis: that is, in the sequel we study the behaviour of $\Pi[\cdot\given X]$ in distribution under $P_{f_0}$, i.e. we assume that there exists a `true' unknown $f_0$ to be estimated, now being back to the assumption of the first paragraph of the section.

\subsection{Function spaces and wavelets}


\textit{Function classes.} Let $L^2 = L^2[0, 1]$ denote the space of square-integrable functions on $[0, 1]$ relative to Lebesgue measure equipped with the $\| \cdot \|_2$-norm. For $f, g \in L^2$, denote $\langle f, g \rangle := \langle f, g \rangle_2 = \int^1_0 fg$. Let $L^\infty = L^\infty [0, 1]$ denote the space of all measurable functions on $[0, 1]$ that are bounded up to a set of Lebesgue measure $0$, equipped with the (essential) supremum norm $\| \cdot \|_\infty$.
The class $\mathcal{C}^\alpha[0, 1]$, $\alpha \in (0, 1]$, of H\"older functions on $[0, 1]$ is
\[ \mathcal{C}^\alpha[0, 1] = \left\{g:[0,1]\to\RR:\quad \sup_{x\neq y} 
\frac{|g(x) - g(y)|}{|x - y|^\alpha}<\infty \right\}. \]
\textit{Haar basis.} The Haar wavelet basis is $\{\vphi, \psi_{lk},  l \geq 0, 0 \leq k < 2^l\}$, where $\vphi = \1_{[0,1]}$ and, for 
  $\psi = - \1_{(0,1/2]} + \1_{(1/2,1]}$ using the convention in \cite{hkpt}, and indices $l\ge 0$, $0 \le k < 2^l$,
 \[ \psi_{lk}(\cdot) = 2^{l/2} \psi(2^l \cdot - k). \]
In this paper our interest is in density functions, that is nonnegative functions $g$ with $\int^1_0 g \vphi = \int^1_0 g = 1$,
so that their first Haar-coefficient is always $1$. So, we will only need to consider the basis functions  $\psi_{lk}$ and simply write informally $(\psi_{lk})$ for the Haar basis. For any $g\in L^2[0,1]$, we denote by $g_{lk}=\psg g,\psi_{lk}\psd$ its Haar wavelet coefficients.

From the definition of the Haar basis it easily follows that a given function in $\mathcal{C}^\alpha[0,1]$  for $\alpha \in (0, 1]$ belongs to the H\"older--type ball, 
for some $R>0$, with $f_{lk}=\psg f,\psi_{lk}\psd$,
\begin{equation} \label{hoelderball}
\mathcal{H}(\al,R):=\{f=(f_{lk}):\ |f_{lk}|\leq R2^{-l(\al+1/2)},\quad \forall l\geq 0, 0\leq k < 2^l\}.
\end{equation}
We also consider the set of densities bounded from above by $m_1<\infty$ and below by $m_0>0$, 
\begin{equation} \label{fset}
 \cF(m_0,m_1)=\{f:\ \int_0^1 f=1,\quad m_0 \le f\le m_1\}.
\end{equation} 
For a given $\alpha > 0$, and $n \geq 1$, let us define 
\begin{equation}\label{minimax4}
\veps^*_{n,\alpha}=\left(\frac{\log n}{n}\right)^{\frac{\alpha}{2\alpha+1}}.
\end{equation}
This is the minimax rate, up to constants, for estimating a density function over a ball $\mathcal{H}(\al,R)$, when the supremum norm is considered as a loss, see \cite{khasminskiibragimov} and \cite{khasminskii}.\\

\noi {\em Notation.} In the sequel $C$ denotes a universal constant whose value only depends on other fixed quantities of the problem and may vary from line to line.\\

\noi {\em Outline.} In Section \ref{sec-pt}, we define tree--type prior distributions for the density estimation model, covering the classical P\'olya trees as particular case, as well as a new class of priors we introduce, called spike--and--slab P\'olya trees, and which is shown to be conjugate in this model. In Section \ref{sec-main}, our main results are derived: first, spike--and--slab P\'olya trees are shown to lead to adaptive minimax posterior rates in the supremum norm in Section \ref{sec-sn}. Next, we derive an adaptive BvM theorem in Section \ref{sec-bvm}. Combining the previous results, adaptive confidence bands are constructed in Section \ref{sec-cs}. A brief discussion of the results can be found in Section \ref{sec-disc}, while Section \ref{sec-proofs} contains the proofs of the main results. \\

\noi {\em Acknowledgements.} The authors would like to thank two referees as well as Thibault Randrianarisoa for insightful comments and suggestions on the paper.

\section{Spike and slab P\'{o}lya trees} \label{sec-pt}

\subsection{Tree prior distributions and densities}  \label{sec-tree}

Here we recall the construction of tree-type distributions following \cite{gvbook17}, Section 3.5, where more background and references can be found.

First let us introduce some notation relative to dyadic partitions. For any fixed indexes $l\geq 0$ and $0 \leq k < 2^l$, the dyadic number $r = k2^{-l}$ can be written in a unique way as $\veps(r) := \veps_1(r) \ldots \veps_l(r)$, its finite expression of length $l$ in base $1/2$ (note that it can end with one or more $0$'s). That is, $\veps_i \in \{0, 1\}$ and \[k2^{-l} = \di \sum_{i=1}^l \veps_i(r)2^{-i}.\]\\ Let $\mathcal{E} := \bigcup_{l \geq 0}\{0,1\}^l\cup \{\emptyset\}$ be the set of finite binary sequences. We write $|\veps| = l$ if $\veps \in  \{0,1\}^l$ and $|\emptyset| = 0$. For $\veps=\veps_1\veps_2\ldots\veps_{l-1}\veps_l$, we also use the notation $\veps'=\veps_1\veps_2\ldots\veps_{l-1}(1-\veps_l)$.

Let us introduce a sequence of partitions $\di \mathcal{I} = \{ (I_\veps)_{|\veps|=l}, l \geq 0\}$ of the unit interval. Here we  consider regular dyadic partitions: this is mostly for simplicity of presentation, and other partitions, based for instance on quantiles of a given distribution, could be considered as well. Set $I_\emptyset = (0, 1]$ and, for any $\veps \in \mathcal{E}$ such that $\veps= \veps(l;k)$  is the expression in base $1/2$ of $k2^{-l}$, set 
\[I_\veps := \left(\frac{k}{2^l},\frac{k+1}{2^l}\right]:=I_k^l. \]	
For any $l \geq 0$, the collection of all such dyadic intervals is a partition of $(0, 1]$.

Suppose we are given a collection of random variables $(Y_{\veps},\,\veps\in\cE)$ with values in $[0,1]$ such that 
\begin{align}
Y_{\veps1} & = 1 - Y_{\veps0},\qquad \forall\, \veps\in\cE, \label{conservative} \\
E[Y_\veps Y_{\veps0} Y_{\veps00}\cdots]&=0
\qquad\qquad\quad \forall\, \veps\in\cE,\ 
\end{align}
Let us then define a random probability measure on dyadic intervals by 
\begin{equation}\label{DEF}
P(I_\veps) = \prod_{j=1}^l Y_{\veps_1\ldots\veps_j}.
\end{equation}
By (a slight adaptation, as we work on $[0,1]$ here, of) Theorem 3.9 in \cite{gvbook17}, 
the measure $P$ defined above extends to a random probability measure on Borel sets of $[0,1]$ almost surely, that we call tree(--type) prior (see below for a justification of this terminology). 
 If $P$ turns out to have a.s. a density with respect to Lebesgue measure (which can be shown to be the case under some conditions on the variables $Y_{\veps}$, see e.g. Theorem 3.16 in \cite{gvbook17}), then the latter prior induces a prior on densities. Another option is to work on a truncated version of $P$, as we explain below.
  
\textit{Paths along the tree.} The distribution of mass in the construction \eqref{DEF}  can be visualised using a tree representation: 
to compute the random mass that $P$ assigns to the subset $I_\veps$ of $[0, 1]$, one follows a binary tree along the expression of $\veps$ : $\veps_1; \veps_1\veps_2,\ldots,\veps_1\veps_2\ldots\veps_l = \veps$. The mass $P(I_\veps)$ is the product of variables $Y_{\veps0}$ or $Y_{\veps1}$ depending on whether one goes `left' ($\veps_j = 0$) or `right' ($\veps_j = 1$) along the tree :\begin{equation}\label{defp}P(I_\veps) = \prod^l_{j=1, \veps_j=0} Y_{\veps_1,\ldots,\veps_{j-1}0} \times \prod^l_{j=1, \veps_j=1}(1 - Y_{\veps_1,\ldots,\veps_{j-1}0}).
\end{equation}
This can be represented graphically, see e.g. \cite{gvbook17}, Figure 3.3 or \cite{mr04}, Figure 4.1.   
A given $\veps = \veps_1, \ldots, \veps_l \in \mathcal{E}$ gives rise to a path $\veps_1\rightarrow \veps_1\veps_2 \rightarrow \veps_1\veps_2 \ldots \veps_l$. We denote $I^{[i]}_\veps := I_{\veps_1\ldots \veps_i}$, for any $i$ in $\{1,\ldots, l\}$. Similarly, denote $$Y^{[i]}_\veps =  Y_{\veps_1\ldots \veps_i}.$$

{\em Truncated tree priors.}  Rather that following the construction in \eqref{DEF} for arbitrary depths  $|\veps|=l$, in practice one may want to `stop' the construction at a large enough (typically $n$--dependent) maximal depth $L$, corresponding to the fact that with a given number of data points $n$, we do not expect information from the data to be present at too deep scales. To do so, we define $P(I_\veps)$ as in \eqref{DEF} for all $\veps$ with $|\veps|\le L$, for $L$ to be chosen, so that $P$ is specified on all dyadic intervals of diameter $2^{-L}$. 
  
There are multiple probability distributions on Borelians of $[0,1]$ that coincide on dyadic intervals $I_{\veps}$ with $P(I_\veps)$ resulting from this  truncated construction. We consider the specific one that is absolutely continuous relatively to the Lebesgue measure on $[0,1]$ with a constant density on each $I_\veps$, $|\veps|=L$. By doing so, if such a truncated tree distribution is taken as a prior on {\em densities}, the posterior distribution is also a distribution on densities, namely a random regular histogram on dyadic intervals of length $2^{-L}$. Specific examples of this construction are given below.

Note that in this construction we do not define $Y_\veps$ for $|\veps|> L$. The same object as above can also be obtained by defining  $Y_\veps$ variables all the way up to infinity by setting $Y_\veps=1/2$ for $|\veps|>L$: it can be checked that both constructions coincide.

\subsection{The case of classical P\'olya trees}

A random probability measure $P$ follows a P\'olya tree distribution $PT(\mathcal{A})$ with parameters $\mathcal{A} =
\{\alpha_\veps ; \veps \in \mathcal{E}\}$ on the sequence of partitions $\mathcal{I}$ if it is a tree prior distribution as defined in Section \ref{sec-tree} with variables $Y_{\veps}$ that, for  $\veps \in \mathcal{E}$, are mutually independent and follow a Beta distribution
\begin{equation}\label{priorbeta}
Y_{\veps0}\sim \text{Beta}(\alpha_{\veps0}, \alpha_{\veps1}).
\end{equation}
A standard assumption is that the parameters $\alpha_\veps$ only depend on the depth $|\veps|$, so that $\alpha_\veps = a_l$ for all $\veps$ with $|\veps|=l$,  any $l \geq 1$, and a sequence $(a_l)_{l\geq 1}$ of positive numbers. The class of P\'olya tree distributions is quite flexible: different behaviours of the sequence of parameters $(\alpha_\veps)_{\veps\in \mathcal{E}}$ give P\'olya trees with possibly very different properties. For instance assuming that $\sum_{l} a_l^{-1}$ converges gives absolutely continuous distributions, while $a_l=2^{-l}$ gives the Dirichlet process with uniform base measure, see \cite{gvbook17} Chapter 3 and \cite{castillo2017} for more details and references.  

As the P\'olya tree prior is characterised by the variables $Y_{\veps0}$, we denote it by $\cP(Y_{\veps0})$.  In case it is a prior  on densities (that is if  $\sum a_l^{-1}$ is finite as recalled above), this leads to the following Bayesian diagram
\[\ba
X|f &\sim& f\\
f &\sim& \cP(Y_{\veps0}),\qquad Y_{\veps0}\sim \text{ Beta}(\al_{\veps0},\al_{\veps1}).
\ea\]
As explained in the previous subsection, one may also consider a truncated version of the P\'olya tree at a certain depth $L$. Posterior convergence rate results for both untruncated and truncated versions are obtained in \cite{castillo2017}. 

\subsection{Spike and slab P\'{o}lya trees}


Let us define a cutoff $L_{max}=\log_2(n)$ and $L\le L_{max}$ to be the largest integer such that
\begin{equation}\label{ellmax}
2^{L} L^2 \leq n.
\end{equation} 


\noindent Let $\Pi$ be the prior on densities generated as follows. 
Consider the tree prior distribution truncated at level $L$ as in \eqref{ellmax} with {\em independent} variables $Y_{\veps0}$ given by
\begin{equation}\label{prior4}
 \varepsilon \in \mathcal{E} \text{ , } Y_{\veps0} \sim (1-\pi_{\veps0})\delta_{\frac{1}{2}}+\pi_{\veps0}\text{Beta}(\alpha_{\veps0},\alpha_{\veps1}),
\end{equation}
with parameters $\alpha_\veps \in \N$ and reals $\pi_\veps>0$  to be chosen, and setting $Y_{\veps1}=1-Y_{\veps0}$. In the sequel, we choose $\pi_\veps$ to depend only on the depth $|\veps|$. In slight abuse of notation, we write $\pi_l=\pi_{\veps}$ for the common value of all $\pi_\veps$ with $|\veps|=l$.   

\textit{Definition.} The truncated tree prior distribution with {\em independent} variables $Y_{\veps0}$ as defined in \eqref{prior4} above with parameters $\al_\veps$, $\pi_\veps$ and cut-off $L$ as in \eqref{ellmax}, is called {\em Spike and Slab P\'olya tree}. It is denoted $\Pi(\al_\veps,\pi_\veps)$, or simply $\cP_S(Y_{\veps0})$ when we want to emphasize the variables $Y_{\veps0}$ defining the tree prior. The name ``spike--and--slab" is chosen for the analogy with the spike--and--slab priors introduced in the literature of high--dimensional statistical models and variable selection, see \cite{mitbeau88}, \cite{georgemc93}, and also used in the context of nonparametrics \cite{js05}, \cite{hoffmann2015}, but with a prior rather featuring a Dirac mass at zero.

This prior is based on an idea of Ghosal and van der Vaart, which is referred as Evenly Split P\'olya tree in their book \cite{gvbook17}, Section 3.7.4. As the authors note, the evenly-split construction enables one to introduce many `even splits' (i.e. $Y_{\veps0}=Y_{\veps1}=1/2$); even all of them except a finite number a.s. if $\pi_{\veps}$ decrease to $0$ fast enough. Our construction slightly differs from theirs in that for simplicity we assume that the splits are always even for $l\ge L$, but the idea is similar. Although the construction in \cite{gvbook17} is mainly motivated from the large support property of the corresponding prior, and the finite number of jumps in the resulting histogram, the prior is also particularly interesting for its connection to wavelet thresholding as we explain now. Indeed,   
 note that using the definition of the Haar basis, the Haar coefficients $f_{lk}$ of a density $f$ can be expressed as (see \eqref{flk})
\begin{equation*}
f_{lk}=\langle f,\psi_{lk} \rangle = 2^{\frac{l}{2}}P(I_\veps)(1-2Y_{\veps0}).
\end{equation*} 
The Spike and Slab P\'olya tree  can therefore be seen as a `thresholding prior', with a thresholding taking place indirectly on the sequence of Haar coefficients of the function via the $Y$ variables, setting Haar coefficients to $0$ as soon as $Y_{\veps0}=1/2$. Therefore there is hope that the posterior distribution can `learn' from the data which coefficients are truly significant in the original unknown signal $f_0$. Our results in the next Section demonstrate that this is indeed the case.

While variables $Y_{\veps0}$ are Beta--distributed in the standard P\'olya tree, they follow a mixture in  the Spike and Slab P\'olya tree $\cP_S$ prior, which  leads to the Bayesian diagram
$$\ba
X|f &\sim& f\\
f &\sim& \cP_S(Y_{\veps0}),\quad \text{ with }\ Y_{\veps0}\sim (1-\pi_{\veps0})\delta_{\frac{1}{2}}+\pi_{\veps0}\text{ Beta}(\al_{\veps0},\al_{\veps1}).
\ea$$
The law $\cP_S$ can be interpreted as a hierarchical prior as follows: on each location $\veps0\equiv (l,k)$ of depth $|\veps0|=l\le L$, one first flips a coin $\gamma_{\veps0}\sim\text{Be}(\pi_{\veps0})$. Second, one sets $Y_{\veps0}\sim (1-\gamma_{\veps0})\delta_{\frac{1}{2}}+\gamma_{\veps0}\text{ Beta}(\al_{\veps0},\al_{\veps1})$. This provides extra flexibility with respect to the standard P\'olya construction, as one can set $Y_{\veps0}$ variables to $1/2$ (thus inducing extra `smoothing') with a certain probability. \\

\noi \textit{The posterior distribution.} 
For $\veps\in\cE$ and parameters $\al_{\veps}>0$, let us introduce the notation
\begin{align}
\label{NiX}
 N_X(I_\veps) & = \di \sum^n_{i=1} \1_{X_i \in I_\veps}, \\
  \alpha_{\veps}(X) & =N_X(I_{\veps})+\alpha_{\veps} \label{alix}.
\end{align} 
Recalling the definition of the Beta function $B(a,b)=(a-1)!(b-1)!/(a+b-1)!$ and $p!=\Ga(p-1)$ the usual factorial, let, 
 with $p_X=p(X,\veps)$, 
\begin{align}
 p_{\alpha_{\veps0}} & =B(\alpha_{\veps0},\alpha_{\veps1})^{-1},\\ 
 p_X & = B(N_X(I_{\veps0})+\alpha_{\veps0},N_X(I_{\veps1})+\alpha_{\veps1})^{-1}.
 \end{align}
 We also let
 \begin{align}
 T_X=&T(\veps,X)=2^{N_X(I_\veps)}\frac{p_{\alpha_{\veps0}}}{p_X} \label{tix}\\
 \tilde{\pi}_{\veps0}=&\frac{\pi_{\veps0}T_X}{(1-\pi_{\veps0})+\pi_{\veps0}T_X}.
 \label{pitilda}
 \end{align}
The following result, proved in Section \ref{sec-proofconj}, shows that the Spike and Slab P\'olya tree prior, as does the classical P\'olya tree prior, is conjugate in the density estimation model. \\
\begin{thm}\label{thmconj}
In the density estimation model, take as prior distribution on the data distribution $P$ a Spike and Slab P\'olya tree prior specified through variables
\[ Y_{\veps0}\sim (1-\pi_{\veps0})\delta_{\frac{1}{2}}+\pi_{\veps0}\bet(\al_{\veps0},\al_{\veps1}).\]
 Then the posterior distribution $P\given X_1,\ldots,X_n$ is again a Spike and Slab P\'olya tree prior which is specified through variables $\tilde{Y}_{\veps0}$  of distribution, 
for every $\varepsilon \in \mathcal{E}$ with $|\veps|\leq L$, 
\begin{equation} \label{posty}
 \tilde{Y}_{\veps0}  \sim (1-\tilde{\pi}_{\veps0})\delta_{\frac{1}{2}}+\tilde{\pi}_{\veps0}\bet(\alpha_{\veps0}(X),\alpha_{\veps1}(X)),
\end{equation}
with $\alpha_\veps(X)$ and $\tilde{\pi}_{\veps0}$ defined in \eqref{alix}--\eqref{tix}.
\end{thm}
Let us note that once the posterior on $\tilde{Y}_{\veps0}$s variables is determined as in \eqref{posty}, then by definition of the tree variables $\tilde{Y}_{\veps1}=1-\tilde{Y}_{\veps0}$ and in particular the marginal distribution of $\tilde{Y}_{\veps1}$ is given by
\[\tilde{Y}_{\veps1}  \sim (1-\tilde{\pi}_{\veps0})\delta_{\frac{1}{2}}+\tilde{\pi}_{\veps0}\bet(\alpha_{\veps1}(X),\alpha_{\veps0}(X)). \]
Note also that if $\pi_{\veps0}=1$ for any $\veps$, meaning that the prior is also a product of Beta variables, one recovers the standard conjugacy for the (truncated at $L$) usual P\'olya tree prior distribution.\\

{\em Spike and Slab P\'olya trees with flat initialisation.}  In the sequel it will be helpful to consider a subclass of Spike and Slab P\'olya tree priors that put zero probability to spikes for the first few layers of the tree, say up to depth $l\le l_0(n)$. That is, we say that such a prior has  flat initialisation up to level $l_0(n)$ if $\pi_{\veps0}=1$ for all $\veps$ with $|\veps|\le l_0(n)$. In the following, $l_0(n)$ will be taken to be a slowly diverging sequence, of much slower order than other cut-offs involved.

\section{Estimation and confidence sets with Spike and Slab P\'olya trees} \label{sec-main}

By definition, we take as prior as above the realisation of the Spike and Slab P\'olya tree $P$ that is absolutely continuous with respect to Lebesgue's measure with density equal to a histogram and histogram heights equal to $2^{|\veps|}P(I_\veps)$. The posterior is, by Theorem \ref{thmconj}, again a Spike and Slab P\'olya tree with density w.r.t. Lebesgue equal to a histogram and histogram heights equal to $2^{|\veps|}\tilde{P}(I_\veps)$. In particular, it induces a posterior on densities that we consider in the main results below.

\subsection{Adaptive supremum-norm convergence rate} \label{sec-sn}
The following result shows that the a posteriori law obtained with a Spike and Slab type P\'olya tree prior concentrates around the true density $f_0$ at minimax rate for the supremum-norm loss.
 
\begin{thm}\label{adapt}
For $\alpha\in (0,1]$ and $R, m_0, m_1>0$, suppose that $f_0$ belongs to $\cH(\al,R)\cap \cF(m_0, m_1)$ as in \eqref{hoelderball}--\eqref{fset}. Let $X_1,\ldots,X_n$ be i.i.d. random variables on $[0,1]$ following $P_{f_0}$. Let  $\Pi$ be the prior on densities induced by a Spike and Slab Polya Tree prior $\Pi(\alpha_\veps, \pi_\veps)$ with the choices, 
\begin{align*}
\alpha_\veps & = a,  \\
\pi_\veps & \, = \, e^{-\kappa l}/\left(\sum_{j=0}^L e^{-\kappa j}\right),\quad l=|\veps|,
\end{align*}
for $a\ge 1$ an integer and $\kappa\ge \kappa_0$, for $\kappa_0=\kappa_0(m_0, m_1)$ a large enough constant. Then for any $M$ large enough, in $P_{f_0}$-probability,
\[  \Pi\left[ \|f-f_0\|_\infty \le M\left(\frac{\log{n}}{n}\right)^{\frac{\alpha}{2\alpha+1}} \given X\right] \to 1. \]
\end{thm}

This theorem is an adaptive version of Theorem 1 of \cite{castillo2017}. There are few results so far dealing with convergence rates in Bayesian density estimation with respect to the supremum-norm loss, among those are the results from \cite{castillo}, \cite{hoffmann2015} and \cite{naulet}. So far, sharp results were known only in the case of regularities larger than $1/2$. Here Theorem \ref{adapt} covers the whole range $(0,1]$ in an adaptive and optimal way. It has been recently hypothesized \cite{yoorr18} that posterior distributions in models far enough apart from Gaussian white noise could perhaps miss adaptive optimal rates in low regularity settings. Theorem \ref{adapt} shows this does not happen in density estimation, and indeed here the proof does not go through using a closeness of the density model with respect to Gaussian white noise (which would possibly require $\al>1/2$ to work if using a form of asymptotic equivalence between the models). Let us now briefly comment on the conditions. The condition that $f_0$ is bounded away from $0$ and $\infty$ is a standard assumption for likelihood--based methods, see also \cite{castillo2017} for more comments on this. The condition that $\al\in(0,1]$ is inherent to the fact that we work with projections of P\'olya  tree densities onto regular dyadic histograms (this projection seems the most natural as the density is defined recursively from successive dyadic partitions of $(0,1]$), and seems difficult to remove at least staying within the type of tree construction considered here, see also the Discussion below. Finally, the exponential decrease in terms of the depth $l$ for the prior $\pi_\veps$ is typical for spike--and--slab type priors, and a similar condition also features in \cite{hoffmann2015} for those in the Gaussian white noise model.

\begin{remark} \label{rem-constants}
The lower bound $\kappa_0$ for the thresholding constant  in Theorem \ref{adapt}  depends on the parameters $m_0, m_1$. A similar phenomenon occurs for classical wavelet thresholding estimators in density estimation, see e.g. \cite{djkp96}  (eq. (9) and Theorem 3), where the thresholding constant has to be chosen large enough in terms of regularity parameters of the true density. There are two simple ways to build a prior that is robust to having no knowledge of $m_0, m_1$. A first option is to take $\kappa=\kappa_n$ to increase to infinity at an arbitrary slow rate. Then the posterior can be seen to contract at optimal rate $\veps_{n,\al}^*$ up to an arbitrarily slow multiplicative factor (a power of $\kappa_n$). A second  option is to set $\pi_{\veps}\propto e^{-\kappa l \log{l}}$ with $\kappa$ fixed; then the posterior rate can be shown to be $\veps_{n,\al}^*(\log\log{n})^\eta$, for some $\eta>0$, which is optimal up to a multiplicative $\log\log(n)^\eta$ factor.
\end{remark}

\begin{remark} \label{rem-rate-flati}
It can be checked that the results of Theorem \ref{adapt} remain unchanged if one takes a spike and slab P\'olya tre prior with flat initialisation up to level $\ell_0(n)$, provided $2^{\ell_0(n)}$ grows to infinity slower than any given power of $n$ (e.g. $\ell_0(n)=\log_2^\eta(n),\ \eta\in (0,1)$ or $\ell_0(n)=\log_2(n)/\log\log(n)$).
\end{remark}

\subsection{Adaptive Bernstein--von Mises theorem} \label{sec-bvm}

To establish a nonparametric Bernstein--von Mises (BvM) result, following \cite{castillo2014} one first finds a space $\mathcal{M}_{0}$ large enough to have convergence at rate $\sqrt{n}$ of the posterior density to a Gaussian process. One can then derive results for some other spaces $\mathcal{F}$ using continuous mapping for continuous functionals $\psi : \mathcal{M}_0 \to \mathcal{F}$. A space that combines well with the supremum norm structure was introduced by \cite{castillo2014} and defined as follows, using an `admissible' sequence $w=(w_l)_{l\geq 0}$ such that $w_l/\sqrt{l} \to \infty$ as $l \to \infty$,
\begin{equation}\label{aimezero}
\mathcal{M}_0=\mathcal{M}_0(w)=\left\lbrace x=(x_{lk})_{l,k}\text{ ; }\lim_{l \to \infty} \max_{0\leq k < 2^l} \frac{|x_{lk}|}{w_l}=0\right\rbrace.
\end{equation}

\noindent Equipped with the norm $\di \|x\|_{\mathcal{M}_0}= \sup_{l\geq 0}\max_{0\leq k < 2^l} |x_{lk}|/w_l$, this is a separable Banach space. In a slight abuse of notation, we  write $f \in \mathcal{M}_0$ if the sequence of its Haar wavelet coefficients belongs to that space $(\langle f, \psi_{lk}\rangle)_{l,k}\in \mathcal{M}_0$ and for a process $(Z(f),\, f\in L^2)$, we write $Z\in\cM_0$ if the sequence $(Z(\psi_{lk}))_{l,k}$ belongs to $\cM_0(w)$ almost surely.  

{\em White bridge process.} For $P$ a probability distribution on $[0,1]$, let us define following \cite{castillo2014} the $P$-white bridge process, denoted by $\mathbb{G}_P$. This is the centered Gaussian process indexed by the Hilbert space $\di L^2(P)=\{f:[0,1]\to \mathbb{R}; \int_0^1f^2dP<\infty\}$ with covariance
\begin{equation}\label{covv}
E[\mathbb{G}_P(f)\mathbb{G}_P(g)]=\int_0^1(f-\int_0^1fdP)(g-\int_0^1gdP)dP.
\end{equation} 
We  denote by $\mathcal{N}$ the law induced by $\mathbb{G}_{P_0}$ (with $P_0=P_{f_0}$) on $\cM_0(w)$. The sequence $(\mathbb{G}_P(\psi_{lk}))_{l,k}$  indeed defines a tight Borel Gaussian variable in $\cM_0(w)$, by Remark 1 of \cite{castillo2014}.

{\em Admissible sequences $(w_l)$.} 
The main purpose of the sequence $(w_l)$ is to ensure that $(\mathbb{G}_P(\psi_{lk}))_{l,k}$ belongs to $\mathcal{M}_0$. Intuitively, without these weights $w_l$, the maximum in \eqref{aimezero} would be over $2^l$ Gaussian variables and of order $\sqrt{2\log(2^l)}=C\sqrt{l}$ which does not tend to $0$ as $l \to \infty$. This also explains why $w_l$ needs to be `just above' $\sqrt{l}$. We refer to \cite{castillo2014}, Section 2.1 and Remark 1, for more background on the choice of $(w_l)$ in the present multiscale setting, and to \cite{castillo2013}, Section 1.2, for a similar discussion in an Hilbert space setting where the targeted loss is the $L^2$--norm.

{\em Bounded Lipschitz metric.} Let $(\mathcal{S},d)$ be a metric space. The bounded Lipschitz metric $\beta_{\mathcal{S}}$ on probability measures of $\mathcal{S}$ is defined as, for any $\mu,\nu$ probability measures of $\mathcal{S}$,

\begin{equation}
\beta_{\mathcal{S}}(\mu,\nu)=\sup_{F;\|F\|_{BL}\leq1}\left|\int_{\mathcal{S}}F(x)(d\mu(x)-d\nu(x))\right|,
\end{equation}

\noi where $F:\mathcal{S}\to \mathbb{R}$ and

\begin{equation}
\|F\|_{BL}=\sup_{x\in \mathcal{S}}|F(x)|+\sup_{x\neq y}\frac{|F(x)-F(y)|}{d(x,y)}.
\end{equation}

\noi This metric metrises the convergence in distribution: $\mu_n \to \mu$ in distribution as $n\to \infty$ if and only if $\beta_{\mathcal{S}}(\mu_n,\mu) \to 0$ as $n\to \infty$, see e.g. \cite{D02}, Theorem 11.3.3.

{\em Recentering the distribution.} To establish the BvM result, one also has to find a suitable way to center the posterior distribution. In this view, denote by $P_n$ the empirical measure
\begin{equation} \label{empmeas}
P_n=\frac{1}{n}\sum_{i=1}^n \delta_{X_i}.
\end{equation}

\noindent Let us also consider $C_n$, which is a smoothed version of $P_n$, defined by

\begin{equation}\label{C_n}
\langle C_n,\psi_{lk} \rangle = \left\{
 \begin{array}{ll}
\langle P_n,\psi_{lk} \rangle & \text{ if } l \leq L\\
\ 0 & \text{ if }  l> L, 
\end{array}\right.
\end{equation}
 where $L$ is the maximal cutoff defined by \eqref{ellmax}.

 We finally introduce $T_n$, which depends on the true parameter $\al$, defined by

\begin{equation}\label{Tn}
\langle T_n,\psi_{lk} \rangle = \left\{ 
\begin{array}{ll}
\langle P_n,\psi_{lk}\rangle,\ &\text{ if } l \leq \cL\\
\ 0,\qquad &\text{ if }  l> \cL, 
\end{array}\right.
\end{equation}

\noi where we defined $\cL=\cL_n(\al)$ to be the integer such that \begin{equation} \label{cell}
2^{\cL} = \lfloor c_0\left(\frac{n}{\log n}\right)^{\frac{1}{1+2\alpha}}\rfloor\end{equation} 
for a suitable large enough constant $c_0>0$ (for $f_0\in\cH(\al,R)$ we see below that the choice $c_0=4R^2$ suits our needs uniformly in $\al\in(0,1])$.

{\em Nonparametric BvM result.} For the following statement, as well as in the next section, we work with spike and slab priors with flat initialisation, as defined at the end of Section \ref{sec-pt}. This is necessary for the next result to hold, see the comments below.

We have the following Bernstein-von Mises phenomenon for $f_0$ in H\"older balls. For $C_n$ as in \eqref{C_n}, we denote by $\tau_{C_n}$ the map $\tau_{C_n}:f\to \sqrt{n}(f-C_n)$.
\begin{thm}\label{BVM}
Let $\mathcal{N}$ denote the distribution induced on $\cM_0(w)$ by the $P_0$--white bridge  $\mathbb{G}_{P_0}$ as defined in \eqref{covv} and let $C_n$ be the centering defined in \eqref{C_n}. Let $l_0=l_0(n)$ be an increasing and diverging sequence with $l_0(n)=o(\log{n})$. For  $m_0, m_1>0$, for $L$ as in \eqref{ellmax}, we consider as prior $\Pi$ a spike and slab P\'olya tree with flat initialisation: for $a\ge 1$ an integer,
\begin{align}
Y_{\veps0}\ \sim  & \ \bet(a,a),\quad \text{ for }|\veps|\leq l_0 \label{priorbvm1}\\
Y_{\veps0}\ \sim &\  (1-\pi_{\veps0})\delta_{1/2}+\pi_{\veps0}\bet(a,a),\quad  \text{ for }l_0<|\veps|\leq L, \label{priorbvm2}
\end{align}
\noi where $\pi_\veps  = e^{-\kappa |\veps|}$ with $\kappa$ chosen as in Theorem \ref{adapt}.
The posterior distribution then satisfies a weak BvM: for every $\al\in(0,1]$ and $R>0$,  
$$\di \sup_{f_0 \in \mathcal{H}(\al,R)\cap \cF(m_0,m_1)} E_{f_0}\left[\beta_{\mathcal{M}_0(w)}(\Pi(\cdot|X)\circ \tau_{C_n}^{-1},\mathcal{N})\right]\to 0,$$
as $n \to \infty$ and for any admissible sequence $w=(w_l)$ with $w_{l_0(n)}/\sqrt{\log{n}} \to \infty$.
\end{thm}
\begin{remark}\label{rem-l}
Recalling that the typical nonparametric cut--off sequence $\cL$ in \eqref{cell} verifies $\cL\sim (\log{n})/(1+2\al)$, assuming $\ell_0(n)=o(\log{n})$ amounts to say that $\ell_0(n)$ does not `interfere' with the nonparametric cut-off $\cL$. Possible choices are, for instance, 
$\ell_0(n)=\log{n}/\log\log{n}$, or $\ell_0(n)=(\log{n})^{1/(1+2\veps)},\, \veps>0$, as in \cite{ray2017} Corollary  3.6.
\end{remark}
Theorem \ref{BVM} is an adaptive BvM result which states that the posterior limiting distribution is  Gaussian; note that, similar to the first such result recently obtained in \cite{ray2017}, one slightly modifies the spike--and--slab prior for the first levels.
Let us comment on this assumption: the prior has a flat initialisation, that is, puts weight $0$ to the spike part $\delta_{1/2}$ for the first levels $l\le l_0(n)$. Without this condition, the corresponding posterior would induce a thresholding of the wavelet coefficients $f_{lk}$ with small $l$, at level of the order $\sqrt{\log{n}/n}$. For some functions $f_0$, this would attribute non-vanishing weights to spikes for some $l\le l_0(n)$, thus preventing the posterior to be Gaussian in the limit. Such a phenomenon was studied in \cite{ray2017}, Proposition 3.7, in the white noise model, for which the author proves that the posterior is not even tight at rate $1/\sqrt{n}$, which is necessary for a weak BvM as in the statement of Theorem \ref{BVM} to hold at rate $1/\rn$, even without talking about normality of the limit. Similar negative results could be proved in the present density estimation setting as well, which is why in all results involving a BvM statement in the sequel we suppose the prior has a flat initialisation.
 
The choice of recentering of the distribution is quite flexible, as it can be checked that the result also holds if one replaces $C_n$ by the posterior mean $\bar{f}_n$ (or also the $\alpha$--dependent centering $T_n$). We note that the `canonical' centering here would be the empirical measure $P_n$ in \eqref{empmeas}, but this choice is not allowed, as $P_n$ does not below to $\cM_0(w)$, hence the need of considering appropriate truncations. For centerings that are truncated versions of  the empirical measure, it is enough  to satisfy the conditions of Theorem 1 of \cite{castillo2014}, which is the case for $L$ as in \eqref{ellmax}.

Using the methods of \cite{castillo2014}, this result  leads to several applications. A first direct implication (this follows from Theorem 5 in \cite{castillo2014}) is the derivation of a confidence set in $\cM_0(w)$. Setting
\begin{equation} \label{csmw}
\cD =\left\{f=(f_{lk}):\  \|f-C_n\|_{\cM_0(w)} \le \frac{R_n}{\rn} \right\},
\end{equation}
where $R_n$ is chosen in such a way that $\Pi[\cD\given X]=1-\ga$, for some $\ga>0$ (or taking the generalised quantile for the posterior radius if the equation has no solution) leads to a set $\cD$ with the following properties: it is a credible set by definition which is also asymptotically a confidence set in $\cM_0(w)$ and the rescaled radius $R_n$ is bounded in probability. 
Other applications are BvM theorems for semiparametric functionals via the continuous mapping theorem and Donsker-type theorems (as in Section 2.2 of \cite{castillo2017}), which do not appear here for the sake of brevity.

\subsection{Adaptive confidence bands} \label{sec-cs}

We now consider the question of deriving adaptive confidence bands (that is, confidence sets for the supremum norm) for the unknown density $f_0$. To do so, we follow the ideas outlined in \cite{castillo2013}--\cite{castillo2014} in the fixed regularity case (note however, that there the regularity was assumed known which is not the case here). Once an adaptive BvM result such as Theorem \ref{BVM} has been derived, one may attempt to intersect the resulting credible set with a ``regularity constraint" (here to be understood as a bound on a norm of an appropriate derivative of $f$, see \eqref{smoothnorm} below). To do so, as here the smoothness $\alpha$ of $f_0$ is unknown, it needs to be `estimated' beforehand. It is well-known that this estimation task is in general too ambitious over typical regularity spaces (such as the H\"older spaces considered here) and this relates to the impossibility in general to construct adaptive confidence sets that have a radius of the order of the minimax rate for adaptive estimation (this was originally noted by Low \cite{low97} in the context of density estimation for the pointwise loss; we refer to chapter 8.3 of the book by Gin\'e and Nickl \cite{ginenicklbook} for a detailed discussion and more references). 
However, this task becomes possible by slightly restricting the set of functions considered, for instance by assuming some form of `self-similarity'. Here we consider Condition 3 from \cite{ginenickl10}, also used in \cite{ray2017}, and which can only be slightly relaxed \cite{bull12}. 

{\em Self-similarity.} Given an integer $j_0>0$, we say that a density $f\in\cH(\alpha,M)$ is {\em self-similar} if, for some constant $\veps>0$, 
\begin{equation} \label{sscond}
\| K_j(f) - f \|_\infty \ge \veps 2^{-j\alpha}\quad \text{for all } j \ge j_0,
\end{equation}
where $K_j(f)=\sum_{l\le j-1}\sum_k \, \langle\psi_{lk},f\rangle \, \psi_{lk}$ is the Haar--wavelet projection  at level $j$.
The class of all such self-similar functions will be denoted by $\cH_{SS}(\alpha,M,\veps)$. 

{\em A pivot density estimator.} 
To carry out the task outlined above of `estimating' the regularity of $f_0$, we use a preliminary estimator based on the P\'olya tree spike and slab posterior distribution $\Pi[\cdot\given X]$ given through \eqref{posty}. For every $\veps\in \cE$, let 
\begin{equation}\label{pmed}
\hat{y}_{\veps0} =\hat{y}_{\veps0}(X)= {\tt{median}}\left\{  (1-\tilde{\pi}_{\veps0})\delta_{\frac{1}{2}}+\tilde{\pi}_{\veps0}\bet(\alpha_{\veps0}(X),\alpha_{\veps1}(X)) \right\}
\end{equation}
 denote the posterior median of the distribution of $\tilde{Y}_{\veps0}$ given in \eqref{posty} and set $\hat{y}_{\veps1}=1-\hat{y}_{\veps0}$. We further set $y_{\veps0}=1/2$ for any $\veps$ such that $|\veps|\ge L$. We denote by 
 \begin{equation}\label{fhat}
 \hat f\equiv (\hat f_{lk})_{l\ge 0, 0\le k<2^l},
 \end{equation}
   the histogram tree-based  density defined through variables $\hat{y}_{\veps0}$ as in \eqref{pmed} by, if $\veps=\veps(l,k)$,
   \begin{equation}
 \hat{f}_{lk} =    \hat{f}_{lk}(X)=
 \begin{cases}
\ 2^{l/2}\left[\prod_{i=1}^{l}\hat y_{\veps_1\cdots\veps_i}\right](1-2\hat y_{\veps0}) & \qquad \text{if } l\le L,\\
\ 0 & \qquad \text{otherwise}.
\end{cases}  
   \end{equation}

{\em Pivot regularity estimator under self-similarity.} We then estimate the ``effective cut-off level" of the posterior median by $\hat{L}=\hat{L}(X)$ defined as
\begin{equation} \label{lhat}
\hat{L} = \max\left\{ l:\ \hat{f}_{lk} \neq 0 \right\}.
\end{equation}
Let us define an associated regularity estimate, with $\hat{L}$ as in \eqref{lhat},
\begin{equation} \label{alhat}
\hat{\al} = \frac12\left[ \frac{1}{\hat{L}} \log_2\left(\frac{n}{\log{n}}\right) - 1 \right]. 
\end{equation}
Let $\al_0\in(0,1]$ be fixed, corresponding to a minimal regularity one wishes to adapt to. The results of Theorem \ref{thmcs} below are uniform for $\al\in[\al_0,1]$, see Remark \ref{rem-unif}. Let us define a norm on wavelet coefficients, for $g$ a bounded function on $(0,1]$ with Haar wavelet coefficients $g_{lk}=\psg g, \psi_{lk}\psd$, and a given regularity index $\be>0$ by, for $d=d(\al_0)=1/\al_0$,
\begin{align}
\|g\|_{\be,L} & = \sup_{l\ge 0} \max_{0\le k<2^l } 2^{l (H(\be,l)+ 1/2)} |g_{lk}|,  \label{smoothnorm} \\
H(\be,l) & = \be \1_{l\le dL} + \al_0 \1_{l> dL},
\end{align}
where 
 $L$ is defined in \eqref{ellmax}. The part for very high levels $l>dL$ in the definition of $H(\be,L)$ ensures that such levels do not cause trouble when checking coverage for the credible set $\cC$ in \eqref{cs} below: any function $g$ with H\"older regularity at least $\al_0$ (which we assume in the statement to follow), in particular the target $f_0$, has wavelet coefficient $(g_{lk})$ such that $2^{l(\al_0+1/2)}|g_{lk}|$ is bounded.

{\em Credible set}. Let us set, for $C_n$ as in \eqref{C_n}, $R_n$ as in \eqref{csmw} and $\hat\al$ as in \eqref{alhat}, 
\begin{equation} \label{cs} 
\cC =\left\{f=(f_{lk}):\  \|f-C_n\|_{\cM_0(w)} \le \frac{R_n}{\rn},\quad \|f\|_{\hat\al,L}\le u_n \right\},
\end{equation}
with $(w_l)$ an admissible sequence and  $u_n\to\infty$ slowly, both to be chosen below.

\begin{thm} \label{thmcs}
Let $\al_0\in(0,1]$ and let $\al\in[\al_0,1]$.  
Let $M\ge 1,\, \ga\in(0,1)$ and $\veps,m_0, m_1 >0$. Let $\Pi$ be a prior as in \eqref{priorbvm1}--\eqref{priorbvm2} in the statement of Theorem \ref{BVM}  with $l_0(n)=(\log{n})/(\log\log{n})$.   
Then the set $\cC$ defined in \eqref{cs} with $w_l=\sqrt{l}(\log{l})$ verifies
\begin{equation}\label{confidence}
  \sup_{f_0\in \cH_{SS}(\alpha,M,\veps)\cap \cF(m_0,m_1)} |P_{f_0}(f_0\in \cC) - (1-\ga) | \to 0,
\end{equation}  
as $n\to\infty$. 
In addition, for every $\al\in[\al_0,1]$ and uniformly over $f_0\in \cH_{SS}(\al,M,\veps)\cap \cF(m_0,m_1)$, the  diameter $|\cC|_\infty=\sup_{f,g\in \cC}\|f-g\|_\infty$ and the credibility of the band  verify, taking $(u_n)$ in \eqref{cs}  to be any sequence such that $\log\log{n} \le u_n\le\log{n}$, as $n\to \infty$, 
\begin{align}
 |\cC|_\infty & = O_{P_{f_0}}((n/\log n)^{-\al/(2\al+1)} u_n),\label{cred_diam}\\
 \Pi[\cC \given X] & = 1-\ga + o_{P_{f_0}}(1). \label{cred}
\end{align} 
Also, for any arbitrary strictly increasing sequence $(u_n)$ with $u_n\le \log\log{n}$, setting $w_l=\sqrt{l}u_l$, one can find a sequence $l_0(n)=o(\log{n})$ that depends on $u_n$ such that  \eqref{confidence}--\eqref{cred_diam}--\eqref{cred} are satisfied.
\end{thm}
Theorem \ref{thmcs} shows that, under self-similarity, for any prescribed confidence level $\ga\in(0,1)$, the set $\cC$ in \eqref{cs} is a credible set of credibility going to $1-\ga$, has confidence level going to $1-\ga$, and is of optimal minimax diameter up to an arbitrary undersmoothing factor $u_n$. This result is comparable, for $\al\in(0,1]$, to the results obtained in the Gaussian white noise model by Ray in \cite{ray2017}. It guarantees that well--chosen credible sets from the posterior distribution have optimal frequentist coverage and diameter. Those are, to the best of our knowledge, the first Bayesian results of this kind in the model of density estimation. The study of frequentist coverage of Bayesian nonparametric credible sets is very recent, we refer to \cite{svv15} for adaptive results in regression and a discussion on the topic, as well as to the recent contributions \cite{svvl2}, \cite{nickl2017}, \cite{cr19}. 

Finally, we note that other ways of constructing confidence sets can be considered. Here we followed the approach suggested by \cite{castillo2014}, but another possibility  would be to proceed following the approach of \cite{svv15}, who directly consider a credible ball in the norm of interest, and then `inflate' it by some large enough factor. Although beyond the scope of the paper, it would be interesting to investigate this approach in the present setting as well.

\begin{remark}\label{rem-unif}
 It can be checked that the results of Theorem \ref{thmcs} are uniform in $\al\in [\al_0,1]$, that is, the constants arising depend only on a lower bound $\al_0$ on the regularity. Also,  the choice of $l_0(n)$ is for simplicity and ensures that $l_0(n)$ is of smaller order than any cut-off $\cL$ arising from \eqref{cell}. For the difference choice  $l_0(n)=(\log{n})^{1/(1+2\veps)},\, \veps>0$, as considered in \cite{ray2017}, Corollary 3.6, the results would be similar, up to slightly different logarithmic terms in the diameter statement for the confidence set.
\end{remark}

\subsection{Discussion} \label{sec-disc}

Let us start this discussion by an analogy. Broadly speaking, P\'olya trees can be seen as density estimation--analogues of Gaussian processes for Gaussian noise regression. Both processes are conjugate in the respective regression and density models. While both processes can be chosen in such a way that they model well $\al$--smooth functions, for fixed $\al$ (with $\alpha\in (0,1]$ for P\'olya trees), they are too `rigid' to allow for statistical adaptation: it can be shown that a mismatch in the specification of the parameters with respect to the true regularity of $f_0$ leads to a suboptimal convergence rate of the corresponding posterior distributions. This difficulty can be overcome by allowing more flexibility in the construction of the prior: for instance, the amount of `smoothing' of a Gaussian process can be drawn at random: one then departs from strictly Gaussian distribution, but it is then possible to get  adaptation, as obtained in  \cite{vvvz09} for adaptation in terms of the quadratic risk. For P\'olya trees, the `evenly-split' idea introduced in \cite{gvbook17} can  be used to build an adaptive posterior, through a procedure we call spike--and--slab P\'olya tree, with the usual Dirac mass at $0$ in the standard spike and slab construction here replaced by a mass at $1/2$. By taking an exponentially decreasing amount of mass outside $1/2$ in the prior distribution, the corresponding posterior distribution is both shown to adapt optimality to the unknown regularity for the supremum norm in terms of rate and in terms of uncertainty quantification.

It is also interesting to put the results in the paper in perspective of wavelet thresholding methods for density estimation \cite{djkp96}. Such estimators keep only empirical wavelet coefficients $\psg \psi_{lk},\mathbb{P}_n\psd$ whose absolute value is larger than some threshold. By doing so, the resulting estimator is not necessarily a density. Note that here the construction respects the density model structure in that the posterior still sits on densities by construction, hence there is no need to project back the estimator onto densities. Constructions of confidence bands in nonparametrics using non--Bayesian approaches have so far  in most cases  been achieved by carefully designing a center and estimate of the radius of the band (for instance using Lepski's method): we refer to \cite{ginenickl10} for such a construction. Here the band $\cC$ in \eqref{cs} is obtained relatively simply on the basis of posterior information only.

Let us now briefly discuss some possible extensions. We note that different ways of deriving adaptation 
could be envisioned: keeping the standard P\'olya tree definition up to a certain level $L$ and then choosing the truncation level $L$ at random would lead to a sieve prior (also called `flat tree' in \cite{cr19}), which would lead, up to a logarithmic factor, to optimal adaptative rates in $L^2$. Note that the Spike and Slab P\'olya tree prior from Theorem \ref{adapt} already gives an adaptive $L^2$ rate up to a logarithmic term (this follows by using $\|\cdot\|_2\le \|\cdot\|_\infty$ on the unit interval; obtaining a sharp $L^2$--rate could be done using a block prior similar in spirit to that built in Section 3.2 of \cite{hoffmann2015}).  Another option, similar in spirit to the construction in \cite{WongMa} would be to draw a binary tree itself at random to specify a skeleton of variables $y_{\veps0}$ equal to $1/2$. Using recently introduced techniques in \cite{cr19}, this should lead to optimal adaptation in $L^\infty$ (up perhaps to logarithmic terms). The work by Ma \cite{lima17} proposes (and derives posterior consistency for) an `adaptive P\'olya tree prior' where the parameters from the Beta variables are themselves given a prior distribution; obtaining posterior rates and inference properties for these flexible priors would be interesting too.  
Concerning the assumption on the regularity  $\alpha\in(0,1]$, it seems difficult to get optimal rates for higher regularities following the original P\'olya tree construction -- note that here we already modify the original prior to get adaptation --, but changing the construction of the prior to allow for more `dependence' (and hence enforce more `regularity) looks promising.  These questions will be considered elsewhere. Finally, in terms of uncertainty quantification, one could possibly slightly extend the notion of self-similarity considered, by proceeding similarly as in \cite{bull12}; it would also be interesting to investigate possible differences between the credible sets based on multiscale norms considered here and sets based on an $L^\infty$--norm posterior quantile. 
  
\section{Proofs of the main results}      \label{sec-proofs}       
\subsection{Preliminaries and notation}      


{\em General notation.}
For a given distribution $P$ with distribution function $F$ and density $f$ on $[0, 1]$, denote $P(B) = F(B) = \int_B f$, for any measurable subset $B$ of $[0, 1]$. In particular under the ``true'' distribution, we denote $P_0(B) = F_0(B) =  \int_B f_0$ as well as $p_\veps=P(I_\veps)$.

For a function $f$ in $L^2$, and $\La$ an integer, denote by $f^{\La}$ the $L^2$-projection of $f$ onto the linear span of all elements of the basis $\{\psi_{lk}\}$ up to level $l = \La$. Also, denote $f^{\La^c}$ the projection of $f$ onto the orthocomplement $\text{Vect}\{\psi_{lk},\, l > \La\}$. In the proofs, we shall use the decomposition $f = f^{\La} + f^{\La^c}$, which holds in $L^2$ and $L^\infty$ under prior and posterior as $f$ is truncated at level $\La$ so has a finite Haar expansion under prior and posterior.

To denote variables $Y$ following the posterior distribution, or to denote the posterior mean, we will henceforth use the following notation.

\textit{1. Tilded notation, posterior distribution.} We denote by $\tilde{P}$ a distribution sampled from the posterior distribution and by $\tilde{Y}$ the corresponding variables $Y$ in \eqref{DEF}. In particular, the variable $\tilde{Y}_{\veps0}$ is distributed following the marginal a posteriori law as specified in \eqref{posty}.

\textit{2. Bar notation, posterior mean.} Let $\bar{f} = \di \int fd\Pi(f |X)$ denote the posterior mean density and $\bar{P}$ the
corresponding probability measure. We use the notation $\bar{Y}$ for the variables defining $\bar{P}$ via \eqref{DEF}. It follows from Theorem \ref{thmconj} that $\bar{P}$ is a tree--type distribution with corresponding variables given by the mean of the variables $\tilde{Y}_{\veps0}$ that is
\begin{equation} \label{yepsbar}
\bar{Y}_{\veps0} = (1-\tilde{\pi}_{\veps0})\frac12 + \tilde{\pi}_{\veps0}
\frac{\al_{\veps0}(X)}{\al_{\veps0}(X)+\al_{\veps1}(X)}.  
\end{equation}

\noi {\em Variables along the tree.} 
Let us recall the definition of the cut-offs $L$  in \eqref{ellmax} and $\cL$  in \eqref{cell}.
 Let, for any $\veps\in\cE$, with $|\veps|=l\le L$, with $L$ as in \eqref{ellmax},
\begin{equation} \label{pyeps} 
p_{0,\veps} = F_0(I_{\veps}),\qquad y_{\veps0} = \frac{F_0(I_{\veps0})}{F_0(I_\veps)}. 
\end{equation}
Further, for $C_0$ to be chosen below, we denote 
\begin{equation}\label{del} 
 \Delta_l=\sqrt{C_0\frac{L2^l}{n}}.
\end{equation}
One notes the key identity, which follows from the definition of the Haar basis: for any bounded density function $f$ corresponding to a distribution $P$ on $[0,1]$,  if \[ f_{lk}=\psg f,\psi_{lk}\psd, \qquad p_\veps=P(I_\veps),\]
we have,  \[ f_{lk}=2^{l/2}(-P(I_{\veps0})+P(I_{\veps1})) = 2^{l/2}(P(I_\veps)-2P(I_{\veps0})).\]  
If $P$ is a tree--type distribution defined through variables $Y_{\veps0}$ as defined in Section \ref{sec-tree}, 
and if $f$ is its corresponding histogram--density, this  leads to
\begin{equation} \label{flk}
f_{lk} = 2^{l/2}p_\veps(1-2Y_{\veps0}),
\end{equation}
using that $p_{\veps0}=p_\veps Y_{\veps 0}$ by definition. 
 By taking the expectation with respect to the posterior distribution $P\given X$, one obtains 
\begin{equation} \label{fbarlk}
\bar{f}_{lk}=2^{l/2}\bar{p}_\veps(1-2\bar{Y}_{\veps0}),
\end{equation}
where we write $\bar{p}_\veps=\bar{P}(I_\veps)$ and $\bar{f}_{lk}=\psg f,\psi_{lk}\psd$ with $\bar{f}=\int f d\Pi(f\given X)$, recalling that the posterior mean $\int Pd\Pi(P\given X)$ is also a tree-type distribution with associated variables $\bar{Y}_{\veps0}$.  
On the other hand, using the definition of $y_{\veps 0}$, one also gets
\begin{equation} \label{folk}
f_{0,lk}=2^{l/2}p_{0,\veps}(1-2y_{\veps0}). 
\end{equation}


\noi {\em An event of high probability.}
For any integer $l$ and for $L$ as in \eqref{ellmax}, let us set
\begin{equation}\label{LAMBDAn}
\Lambda_n(l)^2 := (l + L)\frac{n}{2^l}.
\end{equation}
Define $\mathcal{B}$ the event on the dataspace, for $M$ large enough to be chosen below,
\begin{equation}\label{B}
\cB:=\bigcap_{l = 0}^L \bigcap_{0 \leq k < 2^l} \left\{  |N_X(I^l_k ) - nF_0(I^l_k )| \leq M\Lambda_n(l) \right\}.
\end{equation}
The probability of the complement $P_{f_0}(\cB^c)$ vanishes by Lemma \ref{lemB}.

\subsection{Proof of Theorem \ref{thmconj}} \label{sec-proofconj}

The prior distribution can be interpreted as a prior on the $Y=(Y_{\veps0})_{|\veps|\le L}$ parameters following the spike and slab P\'olya tree distribution as in the statement, themselves defining a distribution on histogram--densities, with random height on the interval $I_\veps$ prescribed by $2^LP(I_\veps)=2^L\prod_{i\le L} Y_{\veps}^{[i]}$, for $|\veps|=L$. Note that in this setting both distributions of $Y$ and of $X\given Y$ then have densities with respect to (finite products of) Lebesgue measure. 
The joint density of $Y,X$ is equal to  
 \[ f(X_1,\ldots,X_n) \di \prod_{\veps:\, |\veps|=0}^{L-1}\left((1-\pi_{\veps0})\1_{\frac{1}{2}}(Y_{\veps0})+\pi_{\veps0}p_{\al_{\veps0}} Y_{\veps0}^{\al_{\veps0}-1}(1-Y_{\veps0})^{\al_{\veps1}-1}\right),\]
where  $\di f(X_1,\ldots,X_n)$ is the product histogram prior density at $(X_1,\ldots,X_n)$ (that is, the density of $X$ given $Y$) given by 
\[\prod_{i=1}^n \prod_{\veps:\, |\veps|=L}(2^LP(I_\veps))^{\1_{X_i\in I_\veps}}=(2^L)^n\prod_{\veps:\, |\veps|=L}P(I_\veps)^{N_X(I_\veps)}.\]
Using the fact that $P$ is tree-induced with variables $Y_{\veps0}$ and noticing that $\di \prod_{\veps:\, |\veps|=0}^{L-1}2^{N_X(I_\veps)}=2^{nL}$, 
\begin{align*}
f(X_1,\ldots,X_n)&= (2^L)^n\prod_{\veps:\, |\veps|=0}^{L-1}Y_{\veps0}^{N_X(I_{\veps0})}(1-Y_{\veps0})^{N_X(I_{\veps1})}\\
&= \prod_{\veps:\, |\veps|=0}^{L-1}2^{N_X(I_\veps)}Y_{\veps0}^{N_X(I_{\veps0})}(1-Y_{\veps0})^{N_X(I_{\veps1})}.
\end{align*}
Bayes' formula now implies that  the posterior density of $Y$ given $X$ is 
\begin{align*} 
&\frac{1}{A_X} \prod_{\veps:\, |\veps|=0}^{L-1}\bigg((1-\pi_{\veps0})2^{N_X(I_\veps)}Y_{\veps0}^{N_X(I_{\veps0})}(1-Y_{\veps0})^{N_X(I_{\veps1})}\1_{\frac{1}{2}}(Y_{\veps0})+\\
&\qquad 2^{N_X(I_\veps)}\pi_{\veps0}p_{\al_{\veps0}} Y_{\veps0}^{N_X(I_{\veps0})+\al_{\veps0}-1}(1-Y_{\veps0})^{N_X(I_{\veps1})+\al_{\veps1}-1}\bigg)\\
& =\frac{1}{A_X} \prod_{\veps:\, |\veps|=0}^{L-1}\left((1-\pi_{\veps0})\1_{\frac{1}{2}}(Y_{\veps0})+2^{N_X(I_\veps)}\pi_{\veps0}p_{\al_{\veps0}} Y_{\veps0}^{N_X(I_{\veps0})+\al_{\veps0}-1}(1-Y_{\veps0})^{N_X(I_{\veps1})+\al_{\veps1}-1}\right),
\end{align*}
where the proportionality constant $A_X$ is given by 
\[ A_X= \prod_{\veps:\, |\veps|=0}^{L-1}\left((1-\pi_{\veps0})+2^{N_X(I_\veps)}\pi_{\veps0}\frac{p_{\al_{\veps0}}}{p_X}\right), \] which concludes the proof of Theorem \ref{thmconj}.
\qed

\subsection{Proof of Theorem \ref{adapt}}

The proof is based on Lemma \ref{Lcomp} below, which shows that the largest level $l$ for which the spike and slab P\'olya tree prior induces non-zero wavelet coefficients $f_{lk}$ does not overshoot the optimal cut--off $\cL$  in \eqref{cell} (provided $c_0$ there is chosen large enough), and induces wavelet coefficients not far from the true signal on levels $l\le \cL$. 
Let $S=S((f_{lk}))=\{(l,k):\ f_{lk}\neq 0\}$ denote the `support' of a given sequence of reals $(f_{lk})$. 

Consider the set of densities 
$\mathbb{A}=\cA_1 \cap \cA_2$, where,  for $\bar{\ga}$ as in Lemma \ref{Lcomp}, 
\begin{align*}
\cA_1 & = \left\{ S \cap \{l>\cL\} = \emptyset \right\}, \\
\cA_2 & = \left\{ \max_{l\le \cL,\, k} |f_{0,lk}-f_{lk}| \le \bar{\ga} \sqrt{\log{n}/n}\right\}.
\end{align*}
For  $\cL$ as in \eqref{cell}, one can decompose the difference $f-f_0$ in three terms,
\begin{equation*} 
f-f_0=(f^{\cL}-f_0^{\cL})+(f^{\cL^c})-(f_0^{\cL^c}).
\end{equation*} 
First, one bounds the pure bias term $f_0^{\cL^c}$ by using that $|f_{0,lk}|\le R2^{-l(1/2+\al)}$ as $f_0\in\cH(\al,R)$,
\begin{align*}
  \|f_0^{\cL^c}\|_\infty & = \Big\| \sum_{l>\cL} \sum_{0\leq k < 2^l} f_{0,lk}\psi_{lk}\Big\|_\infty  \leq\sum_{l>\cL}\left[ \left(\max_{0\leq k < 2^l} |f_{0,lk}| \right)\Big\|\sum_{k=0}^{2^l-1} |\psi_{lk}| \Big\|_\infty \right] \\
  & \lesssim \sum_{l>\cL}2^{-l(1/2+\al)}2^{l/2}\leqa \sum_{l>\cL}2^{-l\al} \lesssim \veps^{*}_{n,\al},
\end{align*}
for $\veps^{*}_{n,\al}$ as in \eqref{minimax4}. Second, one notes that the term $f^{\cL^c}$ is zero on $\cA_1$, as on this set $f_{lk}=0$ for $l>\cL$ by definition. Third, proceeding as for $\|f_0^{\cL^c}\|_\infty$ above,  on $\cA_2$,
\begin{align*}
\| f^{\cL}-f_0^{\cL}\|_\infty & 
\le \sum_{l\le \cL} 2^{l/2} \max_{0\le k< 2^l} |f_{lk} -f_{0,lk}|\\
& \le \sum_{l\le \cL} 2^{l/2} \bar{\ga} \sqrt{\log{n}/n}\leqa \veps^{*}_{n,\al}. 
\end{align*}
Combining the previous bounds gives that $\Pi[\|f-f_0\|_\infty>\mu\veps_{n,\al}^*, f\in \mathbb{A} \given X]$ is zero for $\mu$ a large enough constant. As $E_{f_0}\Pi[\mathbb{A}^c\given X]=o(1)$ by Lemma \ref{Lcomp}, this gives the result.

\subsection{A useful tightness result}

The proof of Theorem \ref{BVM} below uses overall a similar approach as \cite{ray2017}, but the argument has to be adapted to the density estimation model and to the specific Spike and Slab procedure considered here. Also, we need a preliminary result on the posterior concentration in the $\|\cdot\|_{\cM_0}$ norm, that we state and prove in the remaining of this section.

Let us denote, for $l\geq 0$, by $\pi_l$ the projection onto the subspace $V_l$ of $L^2[0,1]$  defined by $$V_l=\text{Vect}\{\psi_{l'k} : 0\leq l' \leq l, 0\leq k < 2^{l'}\}. $$
Similarly we denote by $\pi_{>l}$ the projection onto $\text{Vect}\{\psi_{l'k} : l' > l, 0\leq k < 2^{l'}\}$.

Let us introduce the sets, for some $\gamma>0$,
\begin{equation} \label{mcall}
\mathcal{J}_n(\gamma)=\left\lbrace(l,k):\ |f_{0,lk}|>\gamma \sqrt{\log n/n}\right\rbrace.
\end{equation}
Note, for $f_0\in \cH(\al,R)$ and recalling \eqref{cell}, that $(l,k)\in\mathcal{J}_n(\gamma)$ implies $l\leq \cL$. We will also denote by $S$ the ``support" of $f$ in terms of its wavelet coefficients 
\begin{equation} \label{suppf}
S=\left\lbrace(l,k):\ f_{lk}\neq 0\right\rbrace.
\end{equation}

\begin{thm}\label{tight}
Under the assumptions and prior distribution as in Theorem \ref{BVM}, for every $\eta>0$, $R, m_0, m_1>0$ and $\al \in (0,1]$, there exist $M>0$ and $n_0\ge 1$ such that for every $n\geq n_0$, $$\di \sup_{f_0 \in \mathcal{H}(\al,R) \cap \cF(m_0,m_1)}E_{f_0}\left[\Pi(\| f-f_0\|_{\mathcal{M}_0(\bar{w})}\geq M/\sqrt{n}|X)\right]<\eta,$$
where $\bar{w}=(\bar{w_l})$ verifies $\bar{w_l}\geqa \sqrt{l}$ and $\bar{w}_{l_0(n)}\geqa  \sqrt{\log{n}}$. 
\end{thm}

\begin{proof} 
Let us first briefly comment on the proof. For sufficiently large wavelet levels,  $l> l_0(n)$ (and so  $\bar{w}_l^{-1}\leqa 1/\sqrt{\log{n}}$ by assumption), the sequence $(\bar{w}_l)$ downweights the differences $|f_{lk}-f_{0,lk}|$ by a multiplicative factor at least of the order  $1/\sqrt{\log{n}}$, so that one can use an argument similar to that in the proof of Theorem \ref{adapt} by appealing to Lemma \ref{Lcomp} which gives $|f_{lk}-f_{0,lk}|\leqa \sqrt{\log{n}/n}$ for $l_0(n)\le l\le \cL$.  For $l\le l_0(n)$, it is important to have a spike and slab P\'olya tree prior with flat initialisation. One can then use (a slight modification of) the argument for standard P\'olya trees in \cite{castillo2017}, Theorem 3. In the rest of the proof, $\cM_0=\cM_0(\bar{w})$ with weighting sequence $(\bar{w}_l)$.

Let us fix $\eta>0$ and consider the event, 
for $\cJ_n(\bar\ga)$ and $S$ as in \eqref{mcall}--\eqref{suppf}, and $\bar{\ga}$ to be chosen,
 $$\di A_n=\left\{S^c\cap \mathcal{J}_n(\bar{\gamma})= \emptyset\right\}\cap \left\{S\cap \{l> \cL\}\neq \emptyset\right\}\cap\left\{\max_{(l,k):\,l\leq \cL}|f_{0,lk}-f_{lk}|\leq \bar{\gamma}\sqrt{(\log{n})/n}\right\}.$$
By Lemma \ref{Lcomp}, there exists $\bar{\gamma}>0$ such that for every $\al \in (0,1]$, there exists $B>0$ such that, for every $f_0 \in \mathcal{H}(\al,R)$, we have $E_{f_0}\Pi(A_n^c|X)\lesssim n^{-B}$. Let us write
\[E_{f_0}\Pi(\| f-f_0\|_{\mathcal{M}_0}\geq M/\sqrt{n}|X)\le
Q_1+Q_2+Q_3,\]
where, for $D>0$ a constant to be chosen below, and writing $l_0=l_0(n)$ as shorthand,
$$\ba
Q_1&=&E_{f_0}\Pi(\{\| f-f_0\|_{\mathcal{M}_0}\geq M/\sqrt{n}\} \cap \{\|\pi_{l_0}(f-f_0)\|_{\mathcal{M}_0}\leq D/\sqrt{n}\}\cap A_n |X)\\
Q_2&=&E_{f_0}\Pi(\{\| f-f_0\|_{\mathcal{M}_0}\geq M/\sqrt{n}\} \cap \{\|\pi_{l_0}(f-f_0)\|_{\mathcal{M}_0}> D/\sqrt{n}\}\cap A_n |X)\\
Q_3&=& E_{f_0}\Pi(A_n^c|X). 
\ea$$
The  term $Q_3$ is a $o(1)$ as seen just above. The  term $Q_1$ is bounded by $$E_{f_0}\Pi( \{\|\pi_{>l_0}(f-f_0)\|_{\mathcal{M}_0}\geq (M-D)/\sqrt{n}\}\cap A_n |X).$$ 
As $f_0 \in \mathcal{H}(\al,R)$, and as $2^{\cL} \leqa (n/\log n)^{1/(2\al+1)}$ by definition (see \eqref{cell}), we have $\mathcal{J}_n(\overline{\gamma})\subset \{(l,k):\,l\le C\cL, 0\leq k < 2^l\}$ for some constant $C=C(\al,R)>0$, so that
$$ \sup_{f_0 \in \mathcal{H}(\al,R)}\sup_{l>\cL}\ \bar{w}^{-1}_l \max_k|f_{0,lk}|\leqa \frac{R2^{-\cL(\al+1/2)}}{\sqrt{\cL}}\leqa1/\sqrt{n}.
$$
It now remains to bound the part with the frequencies $l_0<l\leq \cL$. On the event $A_n$, we have
$$\di \sup_{l_0 < l \leq \cL}\bar{w}^{-1}_l \max_k|f_{0,lk}-f_{lk}|\leq \frac{\bar{\gamma}}{\bar{w}_{l_0}}\sqrt{\frac{\log n}{n}}\leqa \frac{\bar{\gamma}}{\sqrt{n}},$$
since our assumptions imply $\bar{w}_{l_0(n)}\geqa \sqrt{\log n}$. This gives us, on $A_n$, $\|\pi_{>l_0}(f-f_0)\|_{\mathcal{M}_0}=O(n^{-1/2})$ for every $f_0 \in \mathcal{H}(\al,R)$. We therefore choose $M=M(\eta)$ to make the term $Q_1$ smaller than $\eta/2$ (in fact, the term even becomes $0$ for $M$ sufficiently large on the event $A_n$). 

The term $Q_2$ is bounded by using Markov's inequality as follows, 
\begin{align*}
Q_2 & \le E_{f_0}\Pi(\sqrt{n}\|\pi_{l_0}(f-f_0)\|_{\mathcal{M}_0}> D |X)
 \le \frac{\sqrt{n}}{D}E_{f_0}E^{\Pi}(\|\pi_{l_0}(f-f_0)\|_{\mathcal{M}_0}|X),
\end{align*}
where $E^\Pi[\cdot\given X]$ denotes the expectation (given $X$) under the posterior distribution. For $T_n$ defined in \eqref{Tn}, one can write  
\begin{align*}
E_{f_0}E^{\Pi}&\left[\|\pi_{l_0}(f-f_0)\|_{\mathcal{M}_0}|X\right]= E_{f_0}E^{\Pi}\left[\max_{l\leq l_0}\frac{1}{\bar{w}_l}\max_k |{f}_{lk}-f_{0,lk}|\given X\right]\\
& \leq \di E_{f_0}E^{\Pi}\left[\max_{l\leq l_0}\frac{1}{\bar{w}_l}\max_k |\langle f-T_n,\psi_{lk}\rangle|\given X\right] + E_{f_0}\left[\max_{l\leq l_0}\frac{1}{\bar{w}_l}\max_k |\langle f_0-T_n,\psi_{lk}\rangle|\right].
\end{align*}
It is now enough to check that both expectations in the last display can be made smaller than $C/\rn$. By taking $D=D(\eta)$ large enough, this will make the term $Q_2$ smaller than $\eta/2$. 
For the first expectation, when $l\leq l_0$ the prior is a Beta distribution with fixed parameters $(a,a)$. This is a similar setting as in the proof of \cite{castillo2017}, Theorem 3, except that there the parameters $a_l\equiv a$ of the Beta increase to infinity as $l2^{2l\al}$, $\al\in(0,1]$, but it is easily checked that this is irrelevant asymptotically for frequencies $l\leq l_0$, so by the same argument as in \cite{castillo2017} (see Lemma 8 and the proof of tightness p. 2092--2094, which establishes the result even for all $l\le \cL$), the first expectation is bounded by $C/\sqrt{n}$. The second expectation can be bounded, following the approach of the proof of Theorem 1 in \cite{castillo2014}, using that $\bar{w}_l\geqa \sqrt{l}$  and with $\kappa>0$ large enough, by 
\[
 \frac{1}{\rn}\cdot\frac{1}{D}\left[\max_{l\leq l_0}\frac{\sqrt{l}}{\bar{w}_l}\right]E_{f_0}\left[\max_{l\leq l_0}\frac{1}{\sqrt{l}}\max_k |\sqrt{n}\langle f_0-T_n,\psi_{lk}\rangle|\right]\]
The term on the right hand side of $1/\rn$ in the last display is bounded by
\begin{align*}
 \frac{\kappa}{D} &+ \frac{1}{D}\int_\kappa^{\infty}P_{f_0}\left(\max_{l\leq l_0}\frac{1}{\sqrt{l}}\max_k |\sqrt{n}\langle f_0-T_n,\psi_{lk}\rangle|>u\right)du\\
\lesssim & \di \frac{\kappa}{D} + \frac{1}{D}\sum_{l\leq l_0,k}\int_\kappa^{\infty}P_{f_0}\left(|\sqrt{n}\langle f_0-T_n,\psi_{lk}\rangle|>\sqrt{l}u\right)du\\
\lesssim & \di \frac{\kappa}{D} + \frac{1}{D}\sum_{l\leq l_0}2^l \int_\kappa^{\infty}e^{-Clu}du \lesssim \di \frac{\kappa}{D} + \frac{1}{D}\sum_{l\leq l_0}e^{-C'\kappa l} \lesssim \frac{1}{D},
\end{align*}
\noindent where the third inequality follows from an application of Bernstein's inequality (see the proof of Theorem 1 in \cite{castillo2014} p. 1959 for details).
This finally gives us that by taking $D=D(\eta)$ large enough the second expectation at stake can be made smaller than $\eta/2$, which concludes the proof of Theorem \ref{tight}.
\end{proof}

\subsection{Proof of Theorem \ref{BVM}}
Within this proof, the shorthand $\cM_0$ refers to $\cM_0(w)$ with $w=(w_l)$ the admissible  sequence from the statement of the result. As a preliminary step for the proof, let us note that starting from $(w_l)$ verifying the stated conditions, one can construct a sequence $(\bar{w}_l)$ such that a) $\bar{w}_l \geqa \sqrt{l}$ and b) $\bar{w}_{l_0(n)}\geqa \sqrt{\log{n}}$ (that is, as required for applying Theorem \ref{tight}) and additionally c) $\bar{w}_l/w_l=o(1)$ as $l\to\infty$. To see this, define $\bar{w}=(\bar{w}_i)$ by
\begin{equation} \label{defwb}
 \bar{w}_i = \sqrt{\log{p}}\qquad \text{if}\quad i\in[l_0(p),l_0(p+1)).
\end{equation} 
Property b) is immediate by definition of $\bar{w}$. Property c) follows from $\bar{w}_i/w_i=\sqrt{\log{p}}/w_i\le \sqrt{\log{p}}/w_{l_0(p)}$ for any $i\in[l_0(p),l_0(p+1))$, using that $(w_i)$ is increasing, so $(\bar{w}_i/w_i)$ is bounded by a sequence that goes to $0$ by assumption $o(1)$, so itself is $o(1)$. Finally, to check a), we have $\bar{w}_i=\sqrt{\log{p}}$ for $i\in[l_0(p),l_0(p+1))$. But $i\le l_0(p+1)\leqa \log(p+1)\leqa \log{p}$ so that $\bar{w_i}\geqa \sqrt{i}$ as requested. 

Let us recall that the centering $C_n$ is defined in \eqref{C_n}. Let us fix $\eta>0$ and denote $\tilde{\Pi}_n=\Pi(\cdot|X)\circ\tau^{-1}_{C_n}$. By the triangle inequality, uniformly over the relevant class of functions, for fixed $l>0$, we have 
\begin{equation}\label{separation}
\beta_{\mathcal{M}_0}(\tilde{\Pi}_n,\mathcal{N})\leq \beta_{\mathcal{M}_0}(\tilde{\Pi}_n,\tilde{\Pi}_n\circ \pi_l^{-1})+\beta_{\mathcal{M}_0}(\tilde{\Pi}_n\circ \pi_l^{-1},\mathcal{N}\circ\pi_l^{-1})+\beta_{\mathcal{M}_0}(\mathcal{N}\circ \pi_l^{-1},\mathcal{N}).
\end{equation}
Let us now look more precisely at the first term on the right hand side of \eqref{separation}. Take a function $F:\mathcal{M}_0\to \mathbb{R}$ such that $\|F\|_{BL}\leq 1$. Let $g_n$ be a random variable following $\tilde{\Pi}_n$ and let $\bar{w}=(\bar{w}_{l})$ be as constructed in \eqref{defwb} above. Let us also consider the events  \[D=\{\| f\|_{\mathcal{M}_0(\bar{w})}\leq M\} \text{ and } D_n=\{\| f-C_n\|_{\mathcal{M}_0(\bar{w})}\leq M/\sqrt{n}\},\] where $M$ is large enough to have $E_{f_0}\left[\Pi(\| f-f_0\|_{\mathcal{M}_0(\bar{w})}\geq M/\sqrt{n}|X)\right]<\eta/9$ as is guaranteed by Theorem \ref{tight}. Then, one can bound from above the difference
\begin{align*}
\lefteqn{\left|\int_{\mathcal{M}_0}Fd\tilde{\Pi}_n-\int_{\mathcal{M}_0}Fd\tilde{\Pi}_n\circ \pi_l^{-1}\right| 
\leq  E^{\tilde{\Pi}_n}\left[|F(g_n)-F(\pi_l(g_n))|\given X\right]}&&\\
&\leq  E^{\tilde{\Pi}_n}\left[|F(g_n)-F(\pi_l(g_n))|(\1_D+\1_{D^c})\given X\right]\leq  E^{\tilde{\Pi}_n}\left[\|g_n-\pi_l(g_n)\|_{\mathcal{M}_0}\|F\|_{BL}\1_D|X\right]+2\tilde{\Pi}_n(D^c \given X)\\
&\leq  E\left[\sup_{l'>l}\frac{1}{w_{l'}}\max_{0\leq k < 2^{l'}}|\sqrt{n}\langle f-C_n,\psi_{l'k}\rangle|\1_{D_n}\given X\right] 
 +2\Pi(\| f-C_n\|_{\mathcal{M}_0(\bar{w})}\geq M/\sqrt{n}\given X).
\end{align*} 
The first term in the last display is bounded from above by 
\begin{align*}
 \left(\sup_{l'>l}\frac{\bar{w}_{l'}}{w_{l'}}\right)E\left[\sup_{l'>l}\frac{1}{\bar{w}_{l'}}\max_{0\leq k < 2^{l'}}|\sqrt{n}\langle f-C_n,\psi_{l'k}\rangle|\1_{D_n}\given X\right] 
\leq \left(\sup_{l'>l}\frac{\bar{w}_{l'}}{w_{l'}}\right)M,
\end{align*}
which is less than $\eta/9$ by choosing $l$ large enough, while the other term is bounded, using the triangle inequality, by
\[ 2\Pi(\| f-f_0 \|_{\mathcal{M}_0(\bar{w})}\geq M/\sqrt{n} \given X)
 +2\Pi(\|f_0-C_n\|_{\mathcal{M}_0(\bar{w})}\geq M/(2\sqrt{n}) \given X).\]
The first term of the last display is bounded in expectation by $\eta/9$ by using Theorem \ref{tight}. 
 The expectation of the second term of the last display corresponds to  a purely frequentist centering and can be handled as in the proof of Theorem 1 of \cite{castillo2014} (with $j_n$ in that result corresponding to the cutoff $L$ if $C_n$ is used as centering) and be made smaller than $\eta/9$ by taking $M$ large enough. Besides, one can note that the result holds when replacing $C_n$ by $T_n$ as in \eqref{Tn} as in that case $j_n$ corresponds to $\cL$ in \eqref{cell} which satisfies the required condition of Theorem 1 of \cite{castillo2014}.
This implies that the first term of \eqref{separation} is smaller than $\eta/3$. 

The last term in \eqref{separation} is as small as desired by taking $l$ large enough:  this corresponds to the fact that the white bridge $\mathbb{G}_P$ of law $\cN$ belongs to the space $\cM_0$ (as established in \cite{castillo2014}, see the proof of Proposition 6 there).

For the middle term on the right hand side of \eqref{separation}, note that $l_0(n)\geq l$ for $n$ large enough. For such $n$, the projected prior onto the first $l$ coordinates is a product of Beta variables and we are exactly in the setting of Theorem 3 in  \cite{castillo2017}, except that the parameters $a_l\equiv a$ of the Beta are constant in our case and do not increase with the depth level. But since $l$ is fixed, the fact that the parameters of our Beta do not depend on $l$ does not change the outcome. Therefore, following the proof of the convergence of the finite-dimensional projections from pages 2089--2091 of \cite{castillo2017}, the middle term can be made smaller than $\eta/3$, which concludes the proof of Theorem \ref{BVM}. \qed

\subsection{Proof of Theorem \ref{thmcs}}

Let us prove the confidence statement first. By definition $\cC=\cD\cap\{f:\, \|f\|_{\hat{\al},L}\le u_n \}$ for $\cD$ as in \eqref{csmw}. For Theorem \ref{BVM} to hold, it suffices that $w_{l_0(n)}/\sqrt{\log{n}} \to \infty$, which holds if one sets $w_l=\sqrt{l}(\log{l})$ and $l_0(n)=(\log{n})/\log\log{n}$. Combining Theorem \ref{BVM} and Theorem 5 in \cite{castillo2014} gives, as noted below the statement of Theorem \ref{BVM}, that $P_{f_0}[f_0\in\cD]\to 1-\ga$, so  it suffices to prove that $P_{f_0}[\|f_0\|_{\hat\al, L}\le u_n]\to 1$ uniformly over $\cH_{SS}(\al,M,\veps)\cap \cF(m_0,m_1)$. 
On the event whose probability is controlled in Lemma \ref{lemreg}, 
\[ \hat{L}\ge \cL+\log_2{c_1}=\frac{1}{1+2\al}\log_2\left(\frac{n}{\log{n}}\right)
+\log_2{c_1}.\]
This implies, by the definition of $\hat\al$, that for any $l\le d L$ (note that one can have $c_1<1$),
\[ l(\hat\al+1/2) = \frac{l}{2}\frac{\log_2(n/\log_2{n})}{\hat{L}}
\le \frac{l}{2}\frac{\log_2(n/\log_2{n})}{\frac{1}{1+2\al}\log_2(n/\log_2{n})+\log_2{c_1}}
\le l\left(\frac12+\al\right)+C\frac{l}{\log_2(n/\log_2{n})}
\]
and for any $l\le dL$ we have  $l\leqa \log_2{n}$, so the last term is less than $l(1/2+\al)+C$ for $l\le dL$, so that for large enough $n$, 
\[ \max_{l\le dL}\max_{0\le k<2^l}  2^{l(\hat\al+1/2)}|f_{0,lk}|\le C'R\le u_n,\]
using $f_0\in\cH_{SS}(\al,M,\veps)\subset \cH(\al,M)$ and $u_n\to\infty$. When $l> dL$, one notes that $|f_{0,lk}|\le M2^{-l(1/2+\al)}\le u_n 2^{-l(1/2+\al_0)}$ for large enough $n$, since $u_n\to\infty$ and $\al\ge \al_0$. From this one deduces that with probability tending to $1$, we have $P_{f_0}( \|f_0\|_{\hat\al,L}\le u_n)\to 1$, which shows the coverage statement.

We now turn to the diameter statement. For any $f,g\in \cC$ and $T_n$ as in 
\eqref{Tn}, 
\begin{align*}
\|f-g\|_\infty & \le \|f-T_n\|_\infty + \|T_n-g\|_\infty \le
2 \sup_{f\in\cC} \|f-T_n\|_\infty.
\end{align*}
In turn, for any $f\in\cC$, one can bound, using the definitions of $\|\cdot\|_{\hat\al,L}$ and $d=1/\al_0$,
\begin{align*}
\|f - & T_n\|_\infty 
 \le  \sum_{l\le \cL} 2^{l/2} \max_{0\le k<2^l} |f_{lk}-T_{n,lk}|
+  \sum_{l=\cL+1}^{dL} 2^{l/2} \max_{0\le k<2^l} |f_{lk}|
+  \sum_{l=dL+1}^\infty 2^{l/2} \max_{0\le k<2^l} |f_{lk}| 
\\
& \le \sum_{l\le \cL} 2^{l/2} w_l w_l^{-1}\max_{0\le k<2^l} |f_{lk}-T_{n,lk}|
+ \sum_{l=\cL+1}^{dL} 2^{-l\hat\al} \max_{0\le k<2^l} 2^{l(1/2+\hat\al)}|f_{lk}| \\
& \ \ \ + \sum_{l=dL+1}^\infty 2^{- l\al_0} \max_{0\le k<2^l} 2^{l(1/2+\al_0)} |f_{lk}| \\
& \le w_{\cL}2^{\cL/2} \|f-C_n\|_{\cM_0(w)} + 2^{-\cL\hat\al} \|f\|_{\hat\al,L}  
+ 2^{-\al_0 d L}  \|f\|_{\hat\al,L} \\
& \le \frac{w_{\cL}}{\sqrt{\cL}}\sqrt{\frac{\cL2^{\cL}}{n}}R_n + 
\left[ 2^{-\cL\hat\al} + 2^{-L} \right] u_n,
\end{align*}
where the last line uses the fact that $f\in\cC$ and the definition of $\cC$, together with the fact that $T_{n,lk}$ and $C_{n,lk}$ coincide for $l\le \cL$. To handle the last term in the previous display, one notes that, proceeding similarly as above (where the upper bound is obtained), for some real constants $c,C$, and any $l\le L$, 
\begin{equation}\label{compahat}
 l(\al+1/2) + c\le l(\hat\al+1/2) \le l(\al+1/2) + C
\end{equation}
from which one deduces that $2^{-\cL\hat\al}\leqa 2^{-\cL\al}$. Also, $2^{-L}u_n\leqa (u_n \log^{2}n)/n\leqa 2^{-\cL\al}$ using $u_n\le \log{n}$. By taking $(w_l/\sqrt{l})$ to be increasing and choosing $w$ in such a way that $w_{L}/\sqrt{L}\le u_n$, one has $w_{\cL}/\sqrt{\cL}\le w_L/\sqrt{L}\le u_n$. These conditions are satisfied for the choice $w_l=\sqrt{l}(\log{l})$ as before, as long as $u_n\ge \log{L}$ which holds if $u_n\ge\log\log{n}$.  One deduces that $\|f-T_n\|_\infty\leqa u_n \veps_{n,\al}(1+R_n)$ which, as $R_n=O_P(1)$ as noted below \eqref{csmw}, gives the desired diameter bound.

Finally, we prove that the considered set has the prescribed credibility. As by definition the first intersecting set $\cD$ in the definition of $\cC$ has credibility  (at least) $1-\ga$, it is enough to prove that $\Pi[\|f\|_{\hat\al,L}\le u_n\given X]$ goes to $1$ in $P_{f_0}$--probability.  This is the same as proving a convergence `rate' (which does not go to $0$ though) of order $u_n$ for the multiscale norm $\|\cdot\|_{\hat\al,L}$. To do so, one proceeds as for the supremum norm result obtained in Theorem \ref{adapt}. In the proof of that result, the following has been derived (for a slightly different prior, see also below the next display): on the event $\mathbb{A}$ from the proof of Theorem 2, whose credibility goes to $1$ in probability, we have 
\[ \max_{l\le \cL,\ k} |f_{0,lk}-f_{lk}|\le \bar{\ga}\sqrt{\log{n}/n} \quad \text{and}\quad  f_{lk}=0\  (l>\cL). \]
We note that the same bounds (up to possibly different constants) hold for the prior \eqref{priorbvm1}--\eqref{priorbvm2}. Indeed, in the regime $l\le l_0(n)$ those bounds are obtained as in the proof of Theorem 1 in \cite{castillo2017} (the only difference is that the prior here is a fixed Beta$(a,a)$ distribution instead of Beta$(a_l,a_l)$ in \cite{castillo2017}, but this only changes the constants involved in the statement in the regime $l\le l_0(n)$), while for $l_0(n)\le l\le \cL$ the result is obtained along the proof of Theorem \ref{adapt} in Lemma \ref{Lcomp}. 
This implies that, for $f$ in the set $\mathbb{A}$ as before, as $f_{lk}=0$ for $l> \cL$ on that set,
\begin{align*}
  &\|f\|_{\hat\al,L}\1_{\mathbb{A}}   \le  
  \max_{l\le \cL} \max_{0\le k<2^l } 2^{l (\hat\al+ 1/2)}
   |f_{lk}| \1_{\mathbb{A}} + 0 \\
  & \leqa \max_{l\le \cL} \max_{0\le k<2^l } 2^{l (\al+ 1/2)}
   |f_{lk}| \1_{\mathbb{A}},
\end{align*}  
using $l\hat\al\le l\al+C$ on the event from Lemma \ref{lemreg} as used also at the beginning on the proof. For $l\le \cL$,
\begin{align*}
2^{l (\al+ 1/2)}
   |f_{lk}| \1_{\mathbb{A}} & \le 2^{l (\al+ 1/2)}
   \left(|f_{lk}-f_{0,lk}|+|f_{0,lk}|\right)\1_{\mathbb{A}}\\
& \leqa 2^{\cL (\al+ 1/2)} \sqrt{\frac{\cL}{n}} + \max_{l\le L} \left( 2^{l (\al+ 1/2)} |f_{0,lk}|\right)\\
& \leqa \sqrt{\frac{n}{\log{n}}}\sqrt{\frac{\cL}{n}} + M,
\end{align*}
which is less than a constant, and hence smaller than $u_n\to\infty$ for large enough $n$. This shows the credibility claim and concludes the proof, up to the last claim: when $(u_n)$ is an arbitrary diverging sequence, one can reproduce the previous arguments as long as one can verify the conditions that: $(w_l/\sqrt{l})$ is nondecreasing and $w_{l_0(n)}/\sqrt{\log{n}} \to \infty$ as well as $w_{L}/\sqrt{L} \le u_n$. To do so, let us set $w_l=\sqrt{l}u_l$. 
Then $(w_l/\sqrt{l})$ is nondecreasing and $w_{L}/\sqrt{L} = u_L\le u_n$ by monotonicity of $(u_n)$. Finally, if one sets $l_0(n)=(\log{n})/u_{\sqrt{\log{n}}}$, we have $w_{l_0(n)}/\sqrt{\log{n}}=u_{l_0(n)}/\sqrt{u_{\sqrt{\log{n}}}}.$ But since $u_{\sqrt{\log{n}}}\le u_n\le \log\log{n}$ by assumption, we have $l_0(n)>\sqrt{\log{n}}\to\infty$ so that $u_{\sqrt{\log{n}}}\le u_{l_0(n)}$ and $w_{l_0(n)}/\sqrt{\log{n}}\ge \sqrt{u_{l_0(n)}}\to\infty$, which concludes the proof of Theorem \ref{thmcs}.

\section{Appendix}

\subsection{Control of population versions}


%

\begin{lem}\label{lemB}
Let $f_0$ be a density bounded with $0<m_0\le f\le m_1<\infty$ on $[0,1]$. Let $\cB$ be the event defined in \eqref{B} and $L$ be as in \eqref{ellmax}. For $n\ge N_0=N_0(m_0)$, and any $M>0$ in the definition of  $\mathcal{B}$ in \eqref{B} such that $M^2\ge 8m_1$, we have
$$P_{f_0}(\mathcal{B}^c)\le 4\exp\left\{- \frac{M^2}{4m_1} L\right\}.$$
\end{lem}
\begin{proof}
The proof is similar to that of Lemma 4 in \cite{castillo2017}, where a stronger control for all $l\ge 0$ is derived: here we only need to control levels $l\le L$ as further levels are truncated in the prior distribution. The difference is that here a precise control of the decrease to $0$ of the probability is needed.  
We briefly re-sketch the argument for completeness: Bernstein's inequality ensures that for $t>0$,
\[ P_{f_0}\left[|N_X(I_k^l) -nF_0(I_k^l)|>t \right]\le 2\exp\left(-\frac{t^2}{2
\left\{ n F_0(I_k^l)+t/3 \right\} }\right). \]
Let $t:=M\La_n(l)$, with $\La_n(l)$ as in \eqref{LAMBDAn}. Then by definition of $\La_n(l)$, as $nF_0(I_k^l)\ge m_0n2^{-l}$, we have for large enough $n$, for all $l\le L$,
\[ t/3 = M\sqrt{(l+L)n/2^l}/3 \le nF_0(I_k^l),  \]
using that $L2^L/n=o(1)$ by definition of $L$ in \eqref{ellmax}. Deduce, with $F_0(I_k^l)\ge m_02^{-l}$, 
\[ P_{f_0}\left[|N_X(I_k^l) -nF_0(I_k^l)|>t \right]\le 2\exp\left(-\frac{t^2}{4
n F_0(I_k^l)  }\right)\le 2\exp\left(\frac{M^2(l+L)}{4m_1}\right). \]
A union bound then gives, setting $\ta=\ta(M)=M^2/(4m_1)$,
\[ P_{f_0}(\mathcal{B}^c)\le 2 \sum_{l\le L}\sum_{0\le k<2^l} e^{-\ta (l+L)}\le 
2e^{-\ta L} \sum_{l\le L} e^{(\log{2}-\ta)l}, \]
which is smaller than $4e^{-\ta L}$ as soon as $\ta\ge 2$, which happens if $M^2\ge 8m_1$.
\end{proof}

\begin{lem}\label{lemnx}
Let  $L$ be as in  \eqref{ellmax} and let $\mathcal{B}$ be the event defined in \eqref{B}. There exists $C_0=C_0(M)>0$ such that for any binary word $\veps\in\cE$ of size $|\veps|=l\le L$, on the event $\cB$,
\[ \left|\frac{N_X(I_{\veps0})+a}{y_{\veps0}(N_X(I_\veps)+2a)}-1\right|
\le C_0 \left(\frac{2^l}{n}2^{l/2}|f_{0,lk}|+\sqrt{\frac{L2^l}{n}}\right). \]
\end{lem}
\begin{proof}
The proof is similar to the proof of Lemma 2 in \cite{castillo2017}. 
We give it below for completeness.  
 One can write $N_X(I_\veps)=nF_0(I_\veps)+\delta_{\veps}$ where $\delta_{\veps}$ is controlled on the event $\cB$ using Lemma \ref{lemB}. Then,
\[ \frac{N_X(I_{\veps0})+a}{y_{\veps0}(N_X(I_\veps)+2a)}-1 = 
\frac{1 + n^{-1}(a+\delta_{\veps0})/F_0(I_{\veps0})}{
1 + n^{-1}(2a+\delta_\veps)/F_0(I_\veps) }. \]
By definition of $\cB$, for $l\le L$ we have $|\delta_{\veps}|\le C\sqrt{nL 2^{-l}}$. Since $L2^{L}=o(n)$ by \eqref{ellmax}, this bound is always of smaller order than $n 2^{-l}\leqa nF_0(I_\veps)$, since $f_0$ is bounded away from $0$. So the denominator in the last display is bounded away from $0$, from which we deduce
\begin{align*}
\left|\frac{N_X(I_{\veps0})+a}{y_{\veps0}(N_X(I_\veps)+2a)}-1\right|
 & \le  C\frac{a}{n}\left|\frac{1}{F_0(I_{\veps0})}-\frac{2}{F_0(I_{\veps})}\right|
 + \frac{C}{n} \Big(\frac{|\delta_{\veps}|}{F_0(I_\veps)}+\frac{|\delta_{\veps0}|}{F_0(I_{\veps0})}\Big)\\
& \le C \frac{ 2^{2l} a}{n} 2^{-l/2}|f_{0,lk}| + C\frac{ 2^{l/2}(nL)^{1/2}}{n},
\end{align*}
on the event $\cB$, where we have used  $|\delta_{\veps}|+|\delta_{\veps0}|\leq C(nL2^{-l})^{1/2}$.
\end{proof}

\subsection{Posterior mean quantities}

Let us recall that $y_{\veps0}, p_{0,\veps}$ are defined in \eqref{pyeps} and that $\bar{Y}_{\veps0}, \bar{p}_\veps$ are defined below \eqref{fbarlk}.
\begin{lem}\label{L2}
Let  $L$ be as in  \eqref{ellmax} and let $\mathcal{B}$ be the event defined in \eqref{B}. Then there exists $C>0$ such that for any $\veps \in \mathcal{E}$ such that $|\veps|=l\le L$, on the event $\mathcal{B}$,  
\begin{align*}
\left|\bar{Y}_{\veps0}-y_{\veps0}\right| & \le C\frac{2^{\frac{l}{2}}}{n}\left(2^l|f_{0,lk}|+\sqrt{nL}\right),\\
 \left|\frac{\bar{p}_\veps}{p_{0,\veps}}-1\right|& \leq C\left(\frac{2^l}{n}+\sqrt{\frac{L 2^{l}}{n}}\right).
 \end{align*}
\end{lem}

\begin{proof}
Recalling $y_{\veps0}=F_0(I_{\veps0})/F_0(I_{\veps})$ and using \eqref{yepsbar} with $\alpha_\veps=a$,
 \[ \left|\frac{\bar{Y}_{\veps0}}{y_{\veps0}}-1\right| \leq (1-\tilde{\pi}_{\veps0})\left|\frac{1}{2y_{\veps0}}-1\right|+\tilde{\pi}_{\veps0}\left|\frac{N_X(I_{\veps0})+a}{y_{\veps0}(N_X(I_\veps)+2a)}-1\right|.\]
 The first term in the last display is bounded by, using the second bound on the posterior weight given in Lemma \ref{lempw}, and writing $\Delta_l=(C_0L2^l/n)^{1/2}$ as in \eqref{del},
 \begin{align*}
  (1-\tilde{\pi}_{\veps0})\left|\frac{1}{2y_{\veps0}}-1\right| 
 & \le \frac{\Delta_l}{y_{\veps0}}\1_{|y_{\veps0}-1/2|\le \Delta_l} + \frac{C}{n}\frac1{y_{\veps0}}\\
& \le  \sqrt{\frac{L2^l}{n}}+\frac{1}{n},
 \end{align*}
  while the second term is bounded using  $\tilde\pi_{\veps0}\le 1$, and Lemma \ref{lemnx},  which leads to the first inequality. 

With $y_i=F_0(I^{[i]}_\veps)/F_0(I^{[i-1]}_\veps)$, we have  $p_{0,\veps}=\prod_{i=1}^l y_i$ and $\bar{p}_\veps=\prod_{i=1}^l w_i$, where, using \eqref{yepsbar}, 
\[ w_i=\bar{Y}_{\veps_1\ldots\veps_i}=\frac{1}{2}(1-\tilde{\pi}^{[i]}_{\veps})+\tilde{\pi}^{[i]}_{\veps}\frac{N_X(I^{[i]}_\veps)+a}{N_X(I^{[i-1]}_\veps)+2a}.\]
Now using the first display of the lemma, and noticing that a similar bound holds for $\bar{Y}_{\veps1}$, we obtain, using  that $f_0\in\cH(\al,R)$,
\[ \left|\frac{w_i}{y_i}-1\right| \le  C\frac{2^{i/2}}{n}\left(2^{i/2-i\al}+\sqrt{nL}\right). \]
An application of Lemma \ref{lemprodsum} (whose conditions are satisfied, as $2^{-i\al}\le 1$) leads to the second inequality of the lemma.
\end{proof}

\subsection{The posterior distribution around its mean}

\begin{lem}\label{L4}
Let  $L$ be as in  \eqref{ellmax} and let $\mathcal{B}$ be the event defined in \eqref{B}. For any $l\le L$ and any $\veps \in \mathcal{E}$ with $|\veps|= l$, let us write, for some $\cC>0$,
 \[B_\veps=4\frac{\sqrt{\cC L}}{\sqrt{nF_0(I_\veps)}}+\frac{4}{nF_0(I_\veps)}.\] Then, on the event $\mathcal{B}$, there exists $D>0$ such that, for $B_\veps, \cC$ as above, for any $\veps$ with $|\veps|=l\le L$,
 \[\Pi(|Y_{\veps0}-\bar{Y}_{\veps0}|>B_\veps|X)
 \le D e^{-\cC L}.\]
\end{lem}

\begin{proof}
Let $Z_X$ be a random variable sampled, given $X$, from a $\bet (\al_{\veps0}(X),\al_{\veps1}(X))$ distribution. The random variable $\tilde{Y}_{\veps0}$ sampled from the posterior distribution of $Y_{\veps0}$ has, by Theorem \ref{thmconj}, the same distribution as $\cY_X=(1-W_X)(1/2)+ W_XZ_X$, where $W_X$ has distribution $\bel(\tilde{\pi}_{\veps0})$, given $X$ and independently of $Z_X$. This implies that, for any $B>0$,
\[ P[|Z_X-E[Z_X\given X]|\le B] \le P[|\cY_X-E[\cY_X\given X]|\le B]=\Pi[|Y_{\veps0}- \bar{Y}_{\veps0}|\le B|X]. \]
To conclude it is enough to control $P[|Z_X-E[Z_X\given X]|> B]$. This is done by applying Lemma \ref{lembeta} to $Z_X\sim \bet (\al_{\veps0}(X),\al_{\veps1}(X))$ with $\phi=\al_{\veps0}(X)$, $\psi=\al_{\veps1}(X)$, so that $\phi+\psi=2a+N_X(I_\veps)$. On  the event $\cB$, we have $N_X(I_\veps)=nF_0(I_\veps)+\delta_\veps$, with $|\delta_\veps|\le M\sqrt{2Ln2^{-l}}$.  By definition of $L$ in \eqref{ellmax}, one obtains $\delta_\veps=o(nF_0(I_\veps))$ (uniformly in $\veps$). Note that, as $nF_0(I_\veps)\geqa n2^{-l}\ge L^2$ goes to infinity and as $f_0$ is bounded away from $0$ and infinity, the conditions of Lemma \ref{lembeta} are satisfied for large enough $n$. Also, on the event $\cB$,
\[ nF_0(I_\veps)/2 \le 2a+N_X(I_\veps) \le 2nF_0(I_\veps). \]
An application of Lemma \ref{lembeta} with $x=\sqrt{L}$ gives the result. 
\end{proof}

\begin{lem} \label{lemeva}
Let $\cB$ be the event defined in \eqref{B}. Let  $\cA=\cA(\cC)$ be the set of histogram densities of $[0,1]$, encoded by the collection of variables $(Y_{\veps})_{|\veps|\le L}$, as follows 
\begin{equation} \label{mcalA}
\mathcal{A}=\bigcap_{\veps:\ |\veps|\le L} \left\{ | Y_\veps - \bar{Y}_\veps|\leq r_\veps\right\},\quad  \text{ with }r_\veps=8\sqrt{\frac{\cC L}{nF_0(I_\veps)}},
\end{equation}
for some $\cC>0$. Then $\Pi[\cA^c\given X] \leqa e^{(\log(2)-\cC)L}$.
\end{lem}
\begin{proof}
Note that, by Lemma \ref{L4}, using that the second term in the definition of $B_\veps$ there is always uniformly of smaller order than the first term, so that it is possible to replace $B_\veps$ by $r_\veps$ as above,
\[ \Pi(\mathcal{A}^c|X)\leqa e^{-\cC L}\sum_{l=0}^{L} 2^l\leqa 2^{L}e^{-\cC L}. \qedhere\] 
\end{proof}

\begin{lem} \label{lempost}
On the event $\cB$ defined in \eqref{B} and on the set $\cA$ of densities $f$ with wavelet coefficients $f_{lk}$ (equivalently characterised through tree--variables $Y_\veps$) defined by \eqref{mcalA}, we have for any $l\le L$ and any $k$,
\begin{equation}\label{aidebar}
|f_{lk}-\bar{f}_{lk}|\lesssim \sqrt{\frac{L}{n}}.
\end{equation}
In particular, for any such $f$ we have $\|f^{\cL}-\bar{f}^{\cL}\|_\infty\lesssim \sqrt{\cL2^{\cL}/n}\lesssim \veps^{*}_{n,\al}$ for $\cL$ as in \eqref{cell}.
\end{lem}
\begin{proof} 
Writing $f_{lk}=2^{l/2}\tilde{p}_\veps(1-2\tilde{Y}_{\veps0})$ under the posterior distribution, and as $\tilde{Y}_{\veps0}$ is a mixture of Beta variables and so is bounded by $1$,  $$\ba
|f_{lk}-\bar{f}_{lk}|&=&\di 2^{\frac{l}{2}}|(\tilde{p}_\veps-\bar{p}_\veps)+2\bar{p}_\veps(\bar{Y}_{\veps0}-\tilde{Y}_{\veps0})+2\tilde{Y}_{\veps0}(\bar{p}_\veps-\tilde{p}_\veps)|\\
&\leq& \di 2^{\frac{l}{2}}\left(|\tilde{p}_\veps-\bar{p}_\veps|+2|\bar{p}_\veps(\bar{Y}_{\veps0}-\tilde{Y}_{\veps0})|+2|\tilde{p}_\veps-\bar{p}_\veps|\right)\\
&\leq& 2^{\frac{l}{2}}\bar{p}_\veps \left(3|\tilde{p}_\veps/\bar{p}_\veps-1|+|\bar{Y}_{\veps0}-\tilde{Y}_{\veps0}|\right).
\ea.$$
Note that $\tilde{Y}_{\veps}$ are by definition the tree-variables $Y$'s corresponding to the posterior distribution, so on the event $\cA$ defined by \eqref{mcalA}, they verify the constraints $|\tilde{Y}_{\veps}-\bar{Y}_\veps|\le r_\veps$. 
By definition, $\tilde p_\veps = \prod_{i=1}^l \tilde{Y}_\veps^{[i]}$ and $\bar{p}_\veps = \prod_{i=1}^l \bar{Y}_\veps^{[i]}$. We wish to apply Lemma \ref{lemprodsum}. To do so, let us note that $|\tilde{Y}_\veps^{[i]}/\bar{Y}_\veps^{[i]}-1|\leqa r_{\veps^{[i]}}$ (recall the notation $r_\veps$ from \eqref{mcalA} for $\veps$ a binary sequence). This comes from the fact that, on the event $\cB$, we have that $\bar{Y}_\veps^{[i]}$ is (thanks to Lemma \ref{L2}) close to $F_0(I_\veps^{[i]})/F_0(I_\veps^{[i-1]})$, which itself is close to $1/2$ and in particular bounded away from $0$ and $1$. An application of Lemma \ref{lemprodsum} leads to, on $\cA$ and $\cB$,
\begin{equation}\label{trocool}
\left|\frac{\tilde{p}_\veps}{\bar{p}_\veps}-1\right|\leqa 
\sum_{i=0}^{l-1}r_{\veps^{[i]}} \leqa \sqrt{\frac{L2^l}{n}}. 
\end{equation}
Using the facts  $\bar{p}_\veps \lesssim 2^{-l}$, and $|\bar{Y}_{\veps0}-\tilde{Y}_{\veps0}|\lesssim \sqrt{L2^l/n}$ on $\mathcal{A}$ as well as \eqref{trocool}, we obtain that, on $\mathcal{A}$ and $\mathcal{B}$,
$|f_{lk}-\bar{f}_{lk}|\lesssim \sqrt{L/n}$. This leads, on $\mathcal{A}$ and $\mathcal{B}$, to $\|f^{\cL}-\bar{f}^{\cL}\|_\infty\lesssim \sqrt{\cL2^{\cL}/n}\lesssim \veps^{*}_{n,\al}$.
\end{proof}

\subsection{Posterior weights control}

\begin{lem}\label{lempw}
Let $L$ be defined as in \eqref{ellmax} and let $\cB$ be the event \eqref{B}. For $l\leq L$, on the event $\mathcal{B}$, for $\veps \in \mathcal{E}$ with $|\veps|= l$, for $\Delta_l$ as in \eqref{del}, there exist constants $C_1,C_2>0$ such that  
\[ 1-\tilde{\pi}_{\veps0}\leq (1-\tilde{\pi}_{\veps0})\1_{\left|y_{\veps0}-1/2\right|\le \Delta_l}+\frac{2^{-l/2}\sqrt{n}}{\pi_{\veps0}} e^{C_1L - C_2nF_0(I_{\veps})\Delta_l^2}\1_{|y_{\veps0}-1/2|> \Delta_l}.\]
Provided $C_0$ in \eqref{del} is chosen large enough, we have, for $\pi_{\veps}$ as in the statement of Theorem \ref{adapt} and for $C_3>0$,
\[1-\tilde{\pi}_{\veps0} \le \1_{|y_{\veps0}-1/2|\leq \Delta_l}+\frac{C_3}{n}\1_{|y_{\veps0}-1/2|> \Delta_l}.\]
\end{lem}
\begin{proof}
It follows from the definition \eqref{pitilda} of $\tilde{\pi}_{\veps0}$, for any set $\cC$, that
\[ 1-\tilde{\pi}_{\veps0}
\le (1-\tilde{\pi}_{\veps0})\1_{\cC}+\frac{1-\pi_{\veps0}}{\pi_{\veps0}}\frac1{T_X}\1_{\cC^c}.
 \]
 Choosing $\cC=\{|y_\veps-1/2|\le \Delta_l\}$, we see that it is enough to bound $T_X$ from below, or equivalently $2^{N_X(I_\veps)}/p_X$ by definition of $T_X$ in \eqref{tix}. To do so, 
let us introduce two numbers $s, q$ defined as
\begin{equation}
s=s_X=N_X(I_{\veps})+2a-2\text{ and }q=q_X=N_X(I_{\veps0})+a-1.
\end{equation}
By definition, one  can now rewrite 
\[\frac{N_X(I_{\veps})+2a-1}{p_X}=\frac{q!(s-q)!}{s!}.\]
Using the non-asymptotic bound   $\sqrt{2\pi}p^{p+1/2}e^{-p}\leq p! \leq \sqrt{2\pi}p^{p+1/2}e^{-p+1/(12p)}$ for any integer $p$, 
\[\frac{N_X(I_{\veps})+2a-1}{p_X} \geq \sqrt{2\pi\frac{q(s-q)}{s}}\frac{(\frac{q}{e})^q(\frac{s-q}{e})^{s-q}}{(\frac{s}{e})^s}e^{-\frac{1}{12s}}.\]
Introducing a multiplicative term $2^s$, let us rewrite
 \[2^s (\frac{q}{e})^q(\frac{s-q}{e})^{s-q}/(\frac{s}{e})^s = \exp\left\{s\left(\frac{q}{s}\log(\frac{2q}{s})+(1-\frac{q}{s})\log(2(1-\frac{q}{s}))\right)\right\}.\]
Denoting by $\bel(a)$ the Bernoulli distribution of parameter $a$ and $\kl(P,Q)$ the Kullback-Leibler divergence between distributions $P$ and $Q$,
\[ 2^s (\frac{q}{e})^q(\frac{s-q}{e})^{s-q}/(\frac{s}{e})^s =\exp\left\{s \kl\left(\bel(\frac{q}{s}),\bel(\frac{1}{2})\right)\right\}.\]
\noi By Lemma \ref{lemkl}, the $\kl$--divergence in the last display is bounded from below by  $\|\bel(\frac{q}{s})-\bel(\frac{1}{2})\|_1^2/2=(2|\frac{q}{s}-\frac{1}{2}|)^2/2$.
 Recalling the definition of $T_X$ in \eqref{tix}, this leads to, for a constant $C=C(a)$,
\[T_X \ge  \frac{C}{\sqrt{s+1}}\sqrt{\frac{q(s-q)}{s(s+1)}}e^{2s\left|\frac{q}{s}-\frac{1}{2}\right|^2-\frac{1}{12s}}. \]
On the event $\mathcal{B}$, we have $N_X(I_\veps)=nF_0(I_\veps)+\delta_\veps$ with $\delta_{\veps}^2 \le nL2^{-l}$. We can also rewrite
\begin{align*}
\frac{q}{s}-\frac{1}{2} & = y_{\veps0}-\frac{1}{2}+y_{\veps0}Z_\veps, 
\end{align*}
where the quantity $Z_\veps$ is defined by
\begin{align*}
Z_\veps & := \left\{ \frac{a-1+\delta_{\veps0}}{nF_0(I_{\veps0})} - 
\frac{2a-2+\delta_{\veps}}{nF_0(I_{\veps})}
\right\}/\left( 1+ \frac{2a-2+\delta_{\veps}}{nF_0(I_{\veps})}\right).
\end{align*}
One can then further use the bounds, with $y_{\veps0}^2 \le 1$,
\[ \left|\frac{q}{s}-\frac{1}{2}\right|^2 \ge \frac12\left(y_{\veps0}-\frac{1}{2}\right)^2-Z_\veps^2. \]
Note that we always have $|\delta_\veps|/(nF_0(I_\veps))\leqa (L2^L/n)^{1/2}$ going to $0$ by the definition of $L$ in \eqref{ellmax}. 
Also, we have
\begin{align*}
Z_\veps^2 & \leqa \left( \frac{\delta_{\veps0}}{nF_0(I_{\veps0})}\right)^2 +
 \left( \frac{\delta_{\veps}}{nF_0(I_{\veps})} \right)^2 + 
 \left(\frac{a-1}{nF_0(I_{\veps0})} - 
\frac{2a-2}{nF_0(I_{\veps})}\right)^2\\
& \leqa \frac{L2^l}{n} + (a-1)^2\left|y_{\veps0}-\frac12\right|^2\frac{1}{(n2^l)^2} \leqa  \frac{L2^l}{n}. 
\end{align*}
By combining the previous bounds and the fact that $s\asymp nF_0(I_\veps)$ (using again the bound on $\delta_\veps$ as above),  one obtains, when $|y_{\veps0}-1/2|>\Delta_l$, 
\begin{align*}
 s\left|\frac{q}{s}-\frac{1}{2}\right|^2 & \ge C_2nF_0(I_\veps)\Delta_l^2 - CnF_0(I_\veps)L2^l/n\\
 & \ge C_2nF_0(I_\veps)\Delta_l^2 - C_1L.
\end{align*}
Combining this inequality with the previous lower bound on $T_X$, one obtains the first inequality of the lemma. To derive the second inequality of the lemma, one notes that, using that $\pi_\veps^{-1}\le e^{C_3l}$ and the rough bound $l\le L$,
\[ \frac{2^{-l/2}\sqrt{n}}{\pi_{\veps0}} e^{C_1L - C_2nF_0(I_{\veps})\Delta_l^2}
\le \sqrt{n}e^{CL} \exp\left\{-C_4n2^{-l}(C_0L2^l/n)\right\},
\]
which is bounded from above by $\sqrt{n} e^{(C-C_4C_0)L}$. 
Provided that $C_0$ is chosen large enough, this can be made smaller than any given power of $n$, which concludes the proof.
%
%

\end{proof}

\begin{lem}\label{lempw2}
Let $L$ be as in \eqref{ellmax} and $\cL$ as in \eqref{cell}, and let $\cB=\cB(M)$ be the event \eqref{B}. 
 Then there exists $C_1>0$ such that, for any $l$ such that $\cL<l\le L$ and $\veps \in \mathcal{E}$ with $|\veps|= l$, on the event $\mathcal{B}$,  
 \[\tilde{\pi}_{\veps0} \le C_1 e^{\xi L}\pi_{l+1}, \]
 where $\xi:=18M^2 /m_0$ and provided we choose $c_0\ge 4R^2$ in \eqref{cell}.
 
In particular, for $\pi_{\veps}$ proportional to $e^{-\kappa |\veps|}$ with $\kappa\ge 4\xi+5\log{2}$, we have, for $\cL\le |\veps|\le L$,  
\[ \tilde{\pi}_{\veps0} \leqa \frac{2^{-l}}{n}. \] 
\end{lem} 
 
\begin{remark} \label{remka} 
By combining Lemma \ref{lemB} and Lemma \ref{lempw2}, we see that it is enough to take $\kappa\ge \kappa_0:=4\cdot18\cdot 8\cdot m_1/m_0+5\log{2}$ for the result of Lemma \ref{lempw2} to hold.
\end{remark}
 
\begin{proof}
As in Lemma \ref{lempw}, we write $s=N_X(I_\veps)+2a-2$ and $q=N_X(I_{\veps0})+a-1$. We also set $\mu:=N_X(I_{\veps0})-N_X(I_{\veps1})$. 
Using the smoothness assumption on $f_0$, one can bound $\mu$ from above on the event $\cB$ by
\[ |\mu|\le n|F_0(I_{\veps0})-F_0(I_{\veps1})|+2M\La_n(l)\le  nR2^{-l(\al+1)}+2M\La_n(l).\]
 Since by assumption $l\ge \cL$, we have $n2^{-l(\al+1)}\le \La_n(l)$ for $l>\cL$,  for $c_0$ in \eqref{cell} large enough (one can take $c_0\ge 4R^2$, and this works uniformly over $\al\in(0,1]$). So for $M\ge 1$,
 \[ \mu\le (2M+1)\La_n(l)\le 3M(nL2^{-l})^{1/2}. \]

Now proceeding as in Lemma \ref{lempw} using Stirling's bounds, one can bound $T_X$ in \eqref{tix} from above by, noting that $q=(s+\mu)/2$ by definition,
\begin{align*} 
T_X&\leqa  \di  \frac{1}{\sqrt{s+1}}\sqrt{\frac{q(s-q)}{s(s+1)}}
\exp\left\{s\left(\frac{q}{s}\log(\frac{2q}{s})+(1-\frac{q}{s})\log(2(1-\frac{q}{s}))\right)\right\}\\
&\leqa   \frac{1}{\sqrt{s+1}}
\exp\left\{s\left(\frac{1}{2}(1+\frac{\mu}{s})\log(1+\frac{\mu}{s})+\frac{1}{2}(1-\frac{\mu}{s})\log(1-\frac{\mu}{s})\right)\right\}\\
& \leqa  \frac{1}{\sqrt{s+1}}e^{\frac{\mu^2}{2s}}.
\end{align*}
\noi As $l\leq L$ with $L$ defined as in \eqref{ellmax}, on the event $\cB$ and for large $n$ we have
\[s\ge  m_0 n2^{-l}-\sqrt{Ln2^{-l}}\ge (m_0/2)n2^{-l}.\]

Combining the previous bounds on $\mu$ and $s$, one obtains that   $\mu^2/(2s)$ is bounded from above by $(18M^2/m_0)L$.  
Using the bound $\tilde{\pi}_{\veps0} \leqa \pi_{\veps0} T_X=\pi_{l+1}T_X$ in \eqref{pitilda}, one obtains the first statement of the lemma. For the second statement,  since $\cL\sim (\log{n})/(1+2\al)$ and $0<\al\le 1$, we have $4\cL\ge(1+\delta)L$, for some $\delta>0$. This gives $\tilde\pi_{\veps0}\leqa e^{(4\xi-\kappa)|\veps|}$ from which  the result follows.
\end{proof}

\subsection{Selection properties of the spike-and-slab P\'olya tree posterior}


\begin{lem}\label{Lcomp}Let $\mathcal{J}_n(\gamma)$ be defined in \eqref{mcall} and $S$ in \eqref{suppf}. There exist $\bar{\gamma}=\bar{\gamma}(R,m_0,m_1)>0$, 
such that for $\cL$ as in \eqref{cell},
\begin{align*}
 E_{f_0}\left[\Pi(S^c\cap \mathcal{J}_n(\bar{\gamma})\neq \emptyset|X)\right]
& \leqa 2^\cL/n, \\
 E_{f_0}\left[\Pi(S\cap \{l> \cL\}\neq \emptyset|X)\right]
 & \leqa 1/n.
\end{align*}
Moreover, for the same $\bar{\gamma}$, we have 
\[ E_{f_0}\left[\Pi\left(\max_{(l,k):\,l\leq L}|f_{0,lk}-f_{lk}|> \bar{\gamma}\sqrt{(\log{n})/n}|X\right)\right]
 \leqa 1/n. \]
\end{lem} 

\begin{proof}
For the first inequality, on $\mathcal{B}$, using the fact that by \eqref{flk}, the wavelet coefficient $f_{lk}=0$ if and only if the corresponding tree variable $Y_{\veps0}=0$ (where the binary sequence $\veps$ is identified with the corresponding dyadic $k2^{-l}$ defined by the pair $(l,k)$), using the correspondence $\veps\equiv (l,k)$, 
 \begin{align*}
 \Pi(S^c\cap \mathcal{J}_n(\bar{\gamma})\neq \emptyset|X)&\leq 
  \sum_{(l,k)\in \mathcal{J}_n(\bar{\gamma})} \Pi(f_{lk}=0|X)
= \sum_{(l,k)\in \mathcal{J}_n(\bar{\gamma})} (1-\tilde{\pi}_{\veps0})\\
&\leq  \sum_{(l,k)\in \mathcal{J}_n(\bar{\gamma})} \left(\1_{|y_{\veps0}-\frac{1}{2}|\leq \Delta_l}+\frac{C}{n}\1_{|y_{\veps0}-\frac{1}{2}|> \Delta_l}\right),
\end{align*}
where the last line uses Lemma \ref{lempw}.
The first term in the last line is in fact $0$, provided $\bar{\gamma}$ chosen large enough, since, again with $\veps\equiv (l,k)$, and $m_0\le f_0\le m_1$,
\[ |y_{\veps0}-\frac{1}{2}|=|\frac{F_0(I_{\veps0})}{F_0(I_\veps)}-\frac12|
=|f_{0,lk}|\frac{2^{-l/2}}{2F_0(I_k^l)}\ge  (2m_1)^{-1} \bar{\ga} 2^{l/2}\sqrt{(\log{n})/n}>\Delta_l,\]
 for $(l,k)\in \mathcal{J}_n(\bar{\gamma})$. This finally gives us that 
$$ \Pi(S^c\cap \mathcal{J}_n(\bar{\gamma})\neq \emptyset|X)\lesssim \di \sum_{l\leq \cL} 2^l/n.$$

For the second inequality of the lemma, on $\mathcal{B}$, using Lemma \ref{lempw2} (and for $\xi, \kappa$ as in that Lemma),
\begin{align*}
 \Pi(S\cap \{l> \cL\}\neq \emptyset|X)&\le  \di \sum_{(l,k):\, l> \cL} \Pi(f_{lk}\neq 0|X)
 = \sum_{(l,k):\, l> \cL} \tilde{\pi}_{\veps0}\\
&\leqa   \sum_{(l,k):\, l> \cL} 2^{-l}/n \leqa 1/n,
\end{align*}
by choosing $\kappa\ge \kappa_0$ as in Remark \ref{remka}.

For the last inequality of the lemma, a union bound gives 
\begin{align*}
\Pi\left[\max_{(l,k):l\leq L}|f_{0,lk}-f_{lk}|> \bar{\gamma}\sqrt{\frac{\log{n}}{n}}|X\right]\\
\leq \sum_{(l,k):l\leq L}\Pi\left[|f_{0,lk}-f_{lk}|> \bar{\gamma}\sqrt{\frac{\log{n}}{n}}|X\right].
\end{align*}


\noi Looking at the expectation under $f_0$ of each term, and in view of recentering by the posterior mean $\bar{f}$, one writes
\[
E_{f_0}\left[\Pi(|f_{0,lk}-f_{lk}|> \bar{\gamma}\sqrt{\frac{\log{n}}{n}}|X)\right]\leq P_{f_0}^n\left(|f_{0,lk}-\bar{f}_{lk}|>\frac{\bar{\gamma}}{2}\sqrt{\frac{\log{n}}{n}}\right)\]
\begin{equation}\label{lastdisplay}+ E_{f_0}\left[\Pi(|f_{0,lk}-f_{lk}|> \bar{\gamma}\sqrt{\frac{\log{n}}{n}}|X)\1_{\{|f_{0,lk}-\bar{f}_{lk}|\leq \frac{\bar{\gamma}}{2}\sqrt{\frac{\log{n}}{n}}\}}\right]
\end{equation} 
Let us first deal with the closeness of $\bar{f}_{lk}$ to $f_{0,lk}$. On the event $\mathcal{B}$, note that
\[ \bar{f}_{lk}=2^{l/2}\bar{p}_\veps(1-2\bar{Y}_{\veps0}),\qquad 
f_{0,lk}= 2^{l/2}p_{0,\veps}(1-2y_{\veps0}),\]
and using Lemma \ref{L2}, on the event $\mathcal{B}$, one can write
\begin{equation*}
|\bar{f}_{lk}-f_{0,lk}|=\left|f_{0,lk}\left(\frac{\bar{p}_\veps}{p_{0,\veps}}-1\right)+2^{1+\frac{l}{2}}\bar{p}_\veps(y_{\veps0}-\bar{Y}_{\veps0})\right| \lesssim |f_{0,lk}|\left(\frac{2^l}{n}+\sqrt{\frac{\cL2^l}{n}}\right)+\sqrt{\frac{\cL}{n}}.
\end{equation*} 
This leads to
 \[ |f_{0,lk}-\bar{f}_{lk}|\leqa 2^{-l(\al+1/2)}\left(\frac{2^l}{n}+\sqrt{\frac{L2^l}{n}}\right)+\sqrt{\frac{\cL}{n}}
 \leqa \sqrt{\frac{\cL}{n}}.\] 
This means that, for $\bar{\ga}$ large enough, the event corresponding to the first term on the right hand side of \eqref{lastdisplay} is a subset of the event $\mathcal{B}^c$. Using Lemma \ref{lemB}, the probability in \eqref{lastdisplay} is therefore bounded by a constant times $e^{-BL}$.

Let us now turn to the last term of \eqref{lastdisplay}, which is bounded by 
\[ E_{f_0}\left[\Pi(|\bar{f}_{lk}-f_{lk}|> \frac{\bar{\gamma}}{2}\sqrt{\frac{\log{n}}{n}}|X)\right].\] 

By Lemma \ref{lempost}, we have $|f_{lk}-\bar{f}_{lk}|\lesssim \sqrt{\cL/n}$ on the set of densities $\mathcal{A}$ from Lemma \ref{lemeva} and on the event $\mathcal{B}$, so the second term is bounded by $P_{f_0}^n(\mathcal{B}^c)+\Pi(\mathcal{A}^c|X) \lesssim e^{-CL}$. This finally leads to \[E_{f_0}\left[\Pi(\max_{(l,k):\,l\leq L}|f_{0,lk}-f_{lk}|> \bar{\gamma}\sqrt{\frac{\log{n}}{n}}|X)\right] \lesssim \sum_{(l,k):\,l\leq L}e^{-CL}\lesssim e^{(\log{2}-C)L}\leqa 1/n,\] provided $C$ is chosen greater than $2\log(2)$, which concludes the proof.
\end{proof}

\subsection{Properties of the posterior coordinate-wise medians}

\begin{lem}  \label{lemreg}        
Let $\hat L$ be defined in \eqref{lhat} and $\cL$ in \eqref{cell}. Let $M\ge 1$ and $\veps, m_0, m_1>0$. For any $\al\in(0,1]$, there exist
$c_1,c_2>0$ such that, as $n\to\infty$, 
\[ \inf_{f_0\in \cH_{SS}(\al,M,\veps)\cap\cF(m_0,m_1)} P_{f_0}\left[ c_12^\cL \le 2^{\hat L} \le c_2 2^\cL \right] \to 1. \]
\end{lem}      
\begin{proof}
For $\hat f$ the density estimator as in \eqref{fhat}, let
\[ \hat{S}=\{ (l,k):\, \hat{f}_{lk}\neq 0 \}. \]
We first prove the intermediate statement, uniformly in $f_0$ as in the statement of the lemma, on the event $\cB$ as in \eqref{B},
\begin{equation} \label{techi}
\hat S \supset \cJ_n(\bar{\ga}),
\end{equation} 
for $\cJ_n(\cdot)$ as in \eqref{mcall} and $\bar{\ga}$ large enough as in Lemma \ref{lempmed}. It is enough to prove that $P_{f_0}[\hat S^c \cap \cJ_n(\bar{\ga})\neq \emptyset]$ goes to $0$ or that, on the event $\cB$, we have $\hat S^c \cap \cJ_n(\bar{\ga})=\emptyset$. By definition of $\hat{f}_{lk}$, if $\hat{f}_{lk}= 0$ then we must have $\hat Y_{lk}= 1/2$. Using Lemma \ref{lempmed}, on the event $\cB$ we have $\hat Y_{lk}\neq 1/2$ for $(l,k)\in\cJ_n(\bar\ga)$, so that indeed $\hat S^c \cap \cJ_n(\bar{\ga})=\emptyset$ on $\cB$. 

By the proof of Proposition 3 in \cite{hoffmann_nickl}, for any $j\ge j_0$, one can find $j_1\in\{j,j+1,\ldots,j+N-1\}$, where $N$ depends only on $\al, R, \veps$ and the wavelet basis (here the Haar basis), but not on $j$, such that, for any $f_0\in \cH_{SS}(\al,M,\veps)$,
\[ \max_{0\le k<2^{j_1}} |f_{0,j_1 k}| \ge C(\veps,\al,R) 2^{-j_1(1/2+\al)}. \]
In particular, this means that choosing $j$ such that $2^j\ge c 2^\cL$ and $c$ small enough, there exists some $c'>0$ and $j_1$ such that $2^{j_1}\ge c'2^{\cL}$ and 
\[  \max_{0\le k<2^{j_1}} |f_{0,j_1 k}| \ge \bar{\ga}\sqrt{\frac{\log{n}}{n}},\]
so that $(j_1,k_1)\in \cJ_n(\bar{\ga})$ for some $k_1$, which combined with \eqref{techi} implies that on the event $\cB$, we must have $2^{\hat L}\ge 2^{j_1}$ so that $2^{\hat{L}}\ge c_1 2^{\cL}$ on $\cB$. 

By the proof of Lemma \ref{lempw2}, we have that on $\cB$, for $l>\cL$, the weight $\tilde{\pi}_{\veps0}$ goes to $0$ with $n$  uniformly in $\cL<l\le L$. This means that for $n$ large enough, the marginal posterior distribution of $Y_{\veps0}$ puts mass going to $1$ on $1/2$ simultaneously for all $l$ inbetween $\cL$ and $L$ (and for $|\veps|>L$, by definition $\tilde Y_{\veps0}=1/2$). This implies that $2^{\hat L}\le 2^{\cL}$ on $\cB$.
\end{proof}

\begin{lem} \label{lempmed}
Let $\hat y$ be as in \eqref{pmed} and let $\cJ_n(\gamma)$ be defined in \eqref{mcall}.  For $\bar{\ga}$ large enough, for any $(l,k)\in \cJ_n(\bar\ga)$, on the event $\cB$,
\[ \hat y_{lk} \neq 1/2. \]
\end{lem}
\begin{proof}

Along the proof of Lemma \ref{Lcomp}, we have obtained, on the event $\cB$, that for all $(l,k)\in\cJ_n(\bar{\ga})$, denoting by $\veps$ the dyadic expression corresponding to $k2^{-l}$ (that is $\veps\equiv (l,k)$), we have $|y_{\veps0}-\frac{1}{2}|>\Delta_l$, which using Lemma \ref{lempw} leads to 
\[ (1-\tilde{\pi}_{\veps0}) \le \frac{C}{n}.\]
The Beta distribution arising in \eqref{pmed}, let us denote it by $L_X$, has expectation  $\bar{b}_{\veps0}:=(N_{\veps0}(X)+a)/(N_{\veps}(X)+2a)$. By Lemma \ref{lemnx}, we have
\[ |\bar{b}_{\veps0}-y_{\veps0}|\le C_0y_{\veps0}\frac{2^l}{n}\left[C2^{-\al l}+\sqrt{\frac{Ln}{2^l}}\right]\leqa \sqrt{\frac{L2^l}{n}}. \]
As recalled above, we have  $|y_{\veps0}-1/2|\geq C\bar{\ga}\sqrt{L2^l/n}$ for $(l,k)\in \mathcal{J}_n(\bar{\gamma})$. Choosing $\bar\ga$ large enough, one deduces, for $(l,k)\in \mathcal{J}_n(\bar{\gamma})$,
\[ |\bar{b}_{\veps0}-1/2| \ge (C/2)\bar{\ga}\sqrt{L2^l/n}.\]  
Also, the variance of the Beta distribution $L_X$ at stake is, by using again concentration of the $N_{\veps}(X)$ variables on the event $\cB$, bounded by  $C2^{l}/n$. By Tchebychev's inequality, this implies that most of the mass of $L_X$ is concentrated on an interval of size $D2^{l/2}/\sqrt{n}$ around its mean, for large $D>0$. In particular, by using the bound otained above on the mean of $L_X$, one deduces that 
the median of $L_X$ is far away from $1/2$ by at least $C'\bar{\ga}\sqrt{L2^l/n}$. Since the weights $(1-\tilde\pi_{\veps0})$ in the mixture in \eqref{pmed} are all smaller than $C/n$ for $(l,k)\in\cJ_n(\bar{\ga})$, we deduce that the median $\hat{y}_{\veps0}$ is far  away from $1/2$ by at least $C'\bar{\ga}\sqrt{L2^l/n}-C/n>0$, so that $\hat{y}_{\veps0}\neq 1/2$. 
\end{proof}

\subsection{Auxiliary lemmas}

Let $\bel(a)$ denote the Bernoulli distribution of parameter $a\in(0,1)$. Let $\kl(P,Q)$ be the Kullback-Leibler divergence between distributions $P$ and $Q$ and $\|P-Q\|_1$ the $L^1$--distance between both.
\begin{lem} \label{lemkl}
For any $a,b\in(0,1)$, we have
\[ \kl(\bel(a),\bel(b))= a\log(a/b)+(1-a)\log((1-a)/(1-b))
\ge \|\bel(a)-\bel(b)\|_1^2/2.\]
\end{lem}
\begin{proof}
This is Pinsker's inequality applied to the Bernoulli distribution.
\end{proof}
\begin{lem}[Lemma 3 of \cite{castillo2017}]\label{lemprodsum} 
Let $\{y_i\}_{1\leq i\leq L}$, $\{w_i\}_{1\leq i\leq L}$ be two sequences of positive real numbers such that there are constants $c_1$, $c_2$ with $$\di \max\limits_{1\leq i\leq L}|\frac{w_i}{y_i}-1|\leq c_1 < 1\text{,   }\sum^L_{i=1}|\frac{w_i}{y_i}-1| \leq c_2 < \infty$$ Then there exists $c_3$ depending on $c_1$, $c_2$ only such that $$ \di \prod^L_{i=1}|\frac{w_i}{y_i}-1|\leq c_3 \sum^L_{i=1}|\frac{w_i}{y_i}-1|.$$
\end{lem}

\begin{lem}[Lemma 6 of \cite{castillo2017}]\label{lembeta}
Let $\phi,\psi$  belong to $(0,\infty)$. Let $Z$ follow a $Beta(\phi,\psi )$ distribution. Suppose, for some reals $c_0$, $c_1$, $$0 < c_0 \leq \phi /(\phi + \psi ) \leq c_1 < 1 \text{ and } \phi \wedge \psi > 8$$
Then there exists $D > 0$ depending on $c_0$, $c_1$ only such that for any $x > 0$, $$P\left[|Z - E[Z]| > \frac{x}{\sqrt{\phi+\psi}}+\frac{2}{\phi + \psi}\right]\leq De^{-\frac{x^2}{4}}.$$
\end{lem}

\bibliographystyle{abbrv}
\bibliography{polbib}
\nocite{*}
\end{document}